\documentclass[12pt, oneside]{article}   
\usepackage{amssymb,amsmath,amsthm,enumerate,hyperref,mathtools}     
\usepackage[dvipsnames]{xcolor} 
\usepackage{ulem} 
\hypersetup{
    colorlinks, 
    linkcolor={red!50!black}, 
    citecolor={blue!50!black}, 
    urlcolor={blue!80!black} 
}
\usepackage[displaymath, mathlines]{lineno}  
\usepackage{mathabx}
\usepackage{graphicx,latexsym,amsfonts}   
\usepackage{lineno}
\usepackage{picture}
\usepackage[english]{babel}

\renewcommand{\emph}{\textsl}

\def\reference#1{\href{#1}{Cliquer ici pour voir une r\'ef\'erence.}} 

\usepackage{tikz, color}
\usetikzlibrary{decorations,decorations.pathmorphing,decorations.pathreplacing}
\usetikzlibrary{arrows,backgrounds,calc}
\usetikzlibrary{shapes.symbols,shapes.geometric}
\usetikzlibrary{decorations.markings}
 \tikzset{->-/.style={decoration={
  markings,
  mark=at position #1 with {\arrow{stealth'}}}, 
  postaction={decorate}}}

\usepackage[longnamesfirst]{natbib} 
 \theoremstyle{plain}
  \newtheorem{theorem}{Theorem}[section]
  \newtheorem{lemma}[theorem]{Lemma}
  \newtheorem{corollary}[theorem]{Corollary}
  \newtheorem{proposition}[theorem]{Proposition}

  \theoremstyle{definition}
  \newtheorem{definition}[theorem]{Definition}
  
  \newtheorem{problem}[theorem]{Problem}
  \newtheorem{example}[theorem]{Example}
  
  \newtheorem{remark}[theorem]{Remark}

\def\st{{\;\vrule height9pt width1pt depth2pt\;}}

\newcommand{\R}{\mathbb{R}}

\def\st{{\;\vrule height9pt width1pt depth1.5pt\;}}

\def\CCC{{\cal C}}

\def\FFF{{\cal F}}
\def\GGG{{\cal G}}
\def\HHH{{\cal H}}

\def\LLL{{\cal L}}

\def\SSS{{\cal S}}

\def\es{\varnothing}
\def\conv{\operatorname{conv}}    

\def\int{\operatorname{int}}
\def\ex{\operatorname{ex}}        

\def\ssbs{\subsetneq{}}    
\def\sbs{\subseteq{}}      
 \def\smu{\setminus{}}

\mathchardef\ordinarycolon\mathcode`\:
\mathcode`\:=\string"8000
\begingroup \catcode`\:=\active
  \gdef:{\mathrel{\mathop\ordinarycolon}}
\endgroup

\def\T#1{{\rm\textsf{(T#1)}}}
\def\V#1{{\rm\textsf{(V#1)}}}
\title{Resolutions of Convex Geometries}

\author{\normalsize
Domenico Cantone\thanks{Department of Mathematics and Computer Science, University of Catania, Catania, Italy. Email: domenico.cantone@unict.it.  In honour of Mario Gionfriddo, for his academic career and recent emeritus professorship}$\:$,
Jean-Paul Doignon\thanks{Department of Mathematics, Universit\'e Libre de Bruxelles, Brussels, Belgium. Email:  doignon@ulb.ac.be \textsl{(corresponding author)}}$\:$,
Alfio Giarlotta\thanks{Department of Economics and Business, University of Catania, Catania, Italy. Email: alfio.giarlotta@unict.it}$\:$,
Stephen Watson\thanks{Department of Mathematics and Statistics, York University, Toronto, Canada. Email: watson@mathstat.yorku.ca}
}

\date{March 2, 2021}

\begin{document}

\makeatletter

\def\@fnsymbol#1{\ensuremath{\ifcase#1\or a\or b\or c \or
   d \or e\or \| \or 7\or 8 \or 9 \else\@ctrerr\fi}}
   
   \maketitle
\makeatother


\noindent \textbf{Abstract}. 
Convex geometries \citep{Edelman_Jamison1985} are finite combinatorial structures dual to union-closed antimatroids or learning spaces.
We define an operation of resolution for convex geometries, which replaces each element of a base convex geometry by a fiber convex geometry.  
Contrary to what happens for similar constructions---compounds of hypergraphs, as in \cite{Chein_Habib_Maurer1981}, and compositions of set systems, as in \cite{Mohring_Radermacher1984}---, resolutions of convex geometries always yield a convex geometry.  

We investigate resolutions of special convex geometries: ordinal and affine. 
A resolution of ordinal convex geometries is again ordinal, but a resolution of affine convex geometries may fail to be affine. 
A notion of primitivity, which generalize the corresponding notion for posets, arises from resolutions: a convex geometry is primitive if it is not a resolution of smaller ones. 
We obtain a characterization of affine convex geometries that are primitive, and compute the number of primitive convex geometries on at most four elements.  
Several open problems are listed. 


\section{Introduction}\label{SECT:intro} 
 
\textsl{Convex geometries}, named after \cite{Edelman_Jamison1985}, are finite mathematical structures that capture combinatorial features of convexity from various settings.  
A `resolution' of convex geometries, as we define it here, is a procedure that builds a new convex geometry from given ones. 
Intuitively, this procedure replaces each element of a given `base' convex geometry by another convex geometry, called a `fiber', and consistently defines a new family of convex sets on the union of the fibers. 

Resolutions of convex geometries are related to a similar construction in the field of choice theory, namely `resolutions of choice spaces',\footnote{\label{FOOT:choice_space}A \textsl{choice space} is a pair $(X,c)$, where $X$ is a nonempty set, and $c \colon 2^X \to 2^X$ maps each nonempty set $A \subseteq X$ to a nonempty subset $c(A) \subseteq A$ (and the empty set to the empty set).} recently introduced by \cite{Cantone_Giarlotta_Watson2020+}.
This is hardly surprising in view of an enlightening result of \cite{Koshevoy1999}, who shows that there is a one-to-one correspondence between convex geometries and special choice spaces, called `path independent'. 
Our construction is also reminiscent of the `compounds of hypergraphs' \citep{Chein_Habib_Maurer1981} or `compositions of set systems' \citep{Mohring_Radermacher1984}. 
There is, however, a salient difference between the two constructions: a composition of convex geometries may fail to be a convex geometry, whereas a resolution of convex geometries is always a convex geometry.

In this paper, we examine resolutions of two special types of convex geometries: ordinal and affine. 
We also investigate `primitive' convex geometries, that is, convex geometries that cannot be obtained as resolutions of smaller convex geometries.
In particular, we show that the notion of a primitive convex geometry generalizes the classical notion of a primitive poset. 
We also perform some computations related to `small' convex geometries, showing that among the $6$ convex geometries on three elements only $1$ is primitive, and exactly $12$ of the $34$ geometries on four elements are primitive. 

The paper is organized as follows.
Section~\ref{SECT_Background} collects preliminary facts on convex geometries.
Section~\ref{SECT:resolutions cg} introduces resolutions of convex geometries, and compares them to  compositions.
Section~\ref{SECT:special types of cg} deals with resolutions of affine and ordinal convex geometries. 
Section~\ref{SECT:cg on at most 4 pts} provides a taxonomy of all primitive convex geometries having at most $4$ elements.
Section~\ref{SECT:open problems} collects several questions and open problems. 


\section{Background on Convex Geometries} \label{SECT_Background}

In this section we provide the basic notions and tools for our analysis.


\subsection{Definitions and Examples} \label{SUBSECT_deff_and_examples}

We start with one of the many equivalent definitions of a convex geometry, a notion that is originally due to \cite{Edelman_Jamison1985}.  
Unless otherwise specified, $X$ is a finite nonempty set, which consists of \textsl{points} or \textsl{elements}, depending on the context.  As usual, $2^X$ denotes the family of all subsets of $X$.

\begin{definition}\label{DEF_convex_geometry} 
A \textsl{convex geometry} on a finite nonempty set $X$ is a collection $\GGG$ of subsets of $X$ satisfying the following three axioms:
\begin{enumerate}[\sf (G1)~]
\item $\es\in\GGG$;
\item $\GGG$ is \textsl{closed under intersection}: if $F$ and $G$ are in $\GGG$, then $F \cap G$ is in $\GGG$;
\item $\GGG$ is \textsl{upgradable}: for any $G$ in $\GGG\setminus\{X\}$, there exists $x$ in $X \setminus G$ such that $G \cup \{x\} \in \GGG$.
\end{enumerate}
Here $X$ is the \textsl{ground set} of $\GGG$, and the sets in $\GGG$ are called \textsl{convex}. 
We slightly abuse terminology, and also call the pair $(X,\GGG)$ a \textsl{convex geometry}.  
A convex geometry $(X,\GGG)$ is \textsl{nontrivial} if $\vert X \vert \geqslant 2$, and \textsl{trivial} otherwise.
\end{definition}

Since $X$ is finite, Axioms~\textsf{(G1)} and \textsf{(G3)} of a convex geometry $(X,\GGG)$ imply that $X = \bigcup \GGG \in \GGG$.
Thus, for any $A \in 2^X$, the family of convex sets $G$ such that $A \sbs G$ always contains $X$.
This fact, along with Axiom~\textsf{(G2)}, ensures the soundness of the following notion: 

\begin{definition} \label{DEF_convex_hull}
Let $(X,\GGG)$ be a convex geometry.
For any $A \in 2^X$, the \textsl{convex hull} of $A$ in $X$ is the smallest convex superset of $A$, that is,
$$
\conv_{\GGG}(A) := \:\bigcap \big\{G \in \GGG : A \sbs G \big\}.
$$ 	
Whenever the family $\GGG$ is clear from context, we shall often simplify notation, and write $\conv(A)$ in place of $\conv_{\GGG}(A)$. 
\end{definition}
 
The next example provides six instances of convex geometries on a three element set.  
It is not difficult to prove that any convex geometry on three elements is isomorphic\footnote{Two convex geometries $(X,\GGG_X)$ and $(Y,\GGG_Y)$ are \textsl{isomorphic} if there is a bijection $\sigma \colon X \to Y$ such that, for each $A \in 2^X$, we have $A \in \GGG_X$ if and only if $\sigma(A) \in \GGG_Y$.} to one of these. (The number of convex geometries for additional sizes of the ground set appears in the  Sequence~\href{https://oeis.org/A224913}{\underline{A224913}} in the OEIS, which considers `(union-closed) antimatroids', that is, the structures that are complementary to convex geometries.)  

\begin{example} \label{EX_all_CGs_on_3_elements_up_to_iso} 
	The following are convex geometries on $X = \{x,y,z\}$:
	\begin{itemize}
		\item $\GGG_1= \big\{\es, \{x\}, \{x,y\},X\big\}$; 
		\item $\GGG_2= \big\{\es, \{x\}, \{y\}, \{x,y\},X\big\}$;
		\item $\GGG_3= \big\{\es, \{x\}, \{x,y\}, \{x,z\},X\big\}$;
		\item $\GGG_4= \big\{\es, \{x\}, \{y\}, \{x,y\}, \{x,z\}, X\big\}$;
		\item $\GGG_5= \big\{\es, \{x\}, \{y\}, \{z\}, \{x,y\}, \{x,z\}, X\big\}$;
		\item $\GGG_6= 2^X$.
	\end{itemize}
\end{example}

There are plenty of settings where convex geometries naturally appear (under various names): see, for instance, \cite{Edelman_Jamison1985}, \cite{Goecke_Korte_Lovasz1989}, \cite{Korte_Lovasz_Schrader1991}, \cite{Doignon_Falmagne1999}, and \cite{Falmagne_Doignon2011}.  
For nice historical overviews with reference to additional settings, see  \cite{Monjardet1985,Monjardet1990b,Monjardet2008}. 

In the next two examples we present two important classes of convex geometries \citep{Edelman_Jamison1985}, namely `ordinal' and `affine'. 
In Section~\ref{SECT:special types of cg}, we shall study in detail these two classes with respect to `resolutions'.

\begin{example}\label{EX_ordinal_Convex_Geometries}\textsf{(Ordinal convex geometries)}
Let $(X,\le)$ be a nonempty finite poset.  
Further, let $\GGG$ be the collection of all \textsl{ideals} of $(X,\le)$, that is,
$$
\GGG := \big\{ G \in 2^X \,\st\, (\forall x, y \in X)\; 
(x \le y \, \wedge \, y \in G) \: \Longrightarrow \: x \in G \big\}.
$$
Then $(X,\GGG)$ is a convex geometry, which is the \textsl{ordinal} convex geometry \textsl{derived} from the partial order $\le$ (or, equivalently, from the poset $(X,\le)$).  
\end{example}

Ordinal convex geometries have a very simple characterization:

\begin{theorem}[\citealp{Edelman_Jamison1985}] \label{THM_chrz ordinal cg}
A convex geometry $(X,\GGG)$ is ordinal if and only if $\GGG$ is closed under union.  
Further, if $(X,\GGG)$ is a convex geometry with $\GGG$ closed under union, then there is exactly one partial order $\le$ on $X$ such that $\GGG$ consists of the ideals of $\le$: in fact, $x \le y$ if and only if $x \in \conv(\{y\})$, for all $x,y \in X$.
\end{theorem}

The partial order $\le$ in Theorem~\ref{THM_chrz ordinal cg} is said to be \textsl{associated} to the ordinal convex geometry $\GGG$.  
Notice that, for all $z,t \in X$, we have
\begin{equation}\label{EQ_order_conv}
z \le t \quad \iff \quad z \in \conv(\{t\}).
\end{equation}

\begin{example}\label{EX_affine_CG}\textsf{(Affine convex geometries)}
Let $X$ be a nonempty finite set of points in some real affine space $\R^d$.  
Moreover, let $\GGG$ be the collection of all sets obtained as intersections of $X$ with convex subsets of $\R^d$, that is,
\begin{equation*} 
  \GGG := \big\{X \cap C \;\st\; C \text{ is a convex subset of } \R^d \big\}.
\end{equation*}
Then $(X,\GGG)$ is a convex geometry, which we call \textsl{affinely embedded}.
The family $\GGG$ is the geometry \textsl{induced} on the subset $X$ of $\R^d$.
For any subset $A$ of $X$, we have $\conv_\GGG(A) = X \cap \conv_\R(A)$, where $\conv_\R$ denotes the convex hull in $\R^d$. 
A convex geometry is \textsl{affine} if it is isomorphic to some affinely embedded convex geometry.
\end{example} 

A special feature of affine convex geometries is that they are \textsl{atomistic}, which means that all their one-element sets are convex \citep{Edelman_Jamison1985}.
The problem of algorithmically characterizing affine convex geometries is nontrivial: on this topic,  see~\cite{Hoffmann_Merckx2018}.

Substructures of arbitrary convex geometries are defined with a procedure similar to the construction of affine convex geometries in Example~\ref{EX_affine_CG}:

\begin{definition} \label{DEF_subgeometries}
Let $(Z,\GGG)$ be a convex geometry, and $\es \neq X \sbs Z$.  
The convex geometry $\HHH$ \textsl{induced on $X$ by $\GGG$} consists of all intersections of $X$ with elements of $\GGG$.\footnote{The proof that $\HHH$ is a convex geometry is straightforward.} 
In this case, we also say that $(X,\HHH)$ is a \textsl{subgeometry} of $(Z,\GGG)$.
\end{definition}


\subsection{Extreme Elements} \label{SUBSECT_extreme pts}

The notion of an `extreme element' of a set is central in the theory of convex geometries \citep[see, for instance,][]{Edelman_Jamison1985}.

\begin{definition} \label{DEF_extreme_elements}
	Let $(X,\GGG)$ be a convex geometry, and $A \in 2^X$. 
	An element $a \in A$ is an \textsl{extreme element of} $A$ if $a \notin \conv(A \setminus \{a\})$. 
	We write $\ex_\GGG(A)$ for the set of extreme elements of $A$, or simply $\ex(A)$ when there is no risk of confusion. 
	The function $\ex_\GGG \colon 2^X \to 2^X$ is the \textsl{extreme operator} on $(X,\GGG)$.  
\end{definition}

Given a convex geometry $(X,\GGG)$, any nonempty set $A \sbs X$ always has at least one extreme element; in particular, $\ex_\GGG(\{x\}) = \{x\}$ for any $x \in X$. 
Observe also that the extreme operator $\ex_\GGG \colon 2^X \to 2^X$ is such that $\ex_\GGG(\es) = \es$ and $\es \neq \ex_\GGG(A) \subseteq A$ for all $A \in 2^X \setminus \{\es \}$; that is, according to Footnote~\ref{FOOT:choice_space} or to Definition~\ref{DEF_choice spaces} below, the pair $\left(X,\ex_\GGG\right)$ is a \textsl{choice space}.  
 
\begin{example} \label{EX_extreme_elements}
	We determine the extreme elements of some sets with respect to the convex geometries of Example~\ref{EX_all_CGs_on_3_elements_up_to_iso}. 
	For $A = \{x,z\}$ and $B = \{y,z\}$, we have:
	$$
	\ex_{\GGG_i}(A) = \left\{
	\begin{array}{lll}
		\{z\} & \hbox{ if~ } 1 \le i \le 4 \\
		A     & \hbox{ if~ } 5 \le i \le 6\,,
	\end{array}
	\right.
	\qquad 
	\hbox{and}
	\qquad 	
		\ex_{\GGG_i}(B) = \left\{
	\begin{array}{lll}
		\{z\} & \hbox{ if~ } 1 \le i \le 2 \\
		B     & \hbox{ if~ } 3 \le i \le 6\,.
	\end{array}
	\right.
	$$	
\end{example}

The next four lemmas collect several properties of the extreme operator, which will be used in later sections. 
Since the first three lemmas are well-known, we shall only prove the fourth.   

\begin{lemma} \label{LEMMA_chrz_extreme_elements}
Let $(X,\GGG)$ be a convex geometry.
The following properties are equivalent for any $A \in 2^X$ and $a \in A$: 
\begin{enumerate}[\quad\rm (1)]
	\item\label{extrelements0} $a \in \ex(A)$;
	\item\label{extrelementsb} $\conv(A \setminus \{a\}) \ssbs \conv(A)$;
	\item\label{extrelementsc} $(\exists G \in \GGG)$ $\big(A \setminus \{a\} \sbs G \; \wedge \; a \notin G\big)$.
\end{enumerate}
\end{lemma}

Theorem~2 in~\cite{Monjardet_Raderanirina2001} yields

\begin{lemma} \label{LEMMA_Monjardet_Rad}
Let $(X,\GGG)$ be a convex geometry.
For any $A \in 2^X$, we have: 
\begin{enumerate}[\quad \rm(i)]
  \item \label{item_MR_1} $\ex(\conv(A)) = \ex(A)$;
  \item \label{item_MR_2} $\conv(\ex(A)) = \conv(A)$.
\end{enumerate}
\end{lemma}

The following properties are given in \citet{Edelman_Jamison1985}: 

\begin{lemma} \label{LEMMA_extreme_for_convex}
Let $(X,\GGG)$ be a convex geometry. 
For any $G \in \GGG$ and $E \in 2^X$, we have: 
\begin{enumerate}[\quad \rm(i)]
	\item $\ex(G) = \{g\in G \,\st\, G\setminus\{g\}\in\GGG\}$; 
	\item $G = \conv(\ex(G))$;
    \item $E \sbs \ex(G) \; \implies \; G \setminus E \in \GGG$.
\end{enumerate}
\end{lemma}

Finally, we prove 

\begin{lemma} \label{LEMMA_alpha_and_rho_for_CG}
Let $(X,\GGG)$ be a convex geometry.  
For any $A,B \in 2^X$, we have:
\begin{enumerate}[\quad \rm(i)]
	\item $A \sbs B  \;\implies\; A \cap \ex(B) \sbs \ex(A)$; 
	\item $\ex(A) \cap B \neq \es \;\implies\; \ex(A \cup B) \cap B \neq \es$.
\end{enumerate}
\end{lemma}

\begin{proof}
To prove (i), suppose $A \subseteq B$, and let $a\in A \cap \ex(B)$.  
By Lemma~\ref{LEMMA_chrz_extreme_elements}, there exists $G$ in $\GGG$ such that $B\setminus\{a\}\sbs G$ and $a\notin G$.  
Then $A \setminus \{a\} \sbs G$ and $a\notin G$, hence $a \in \ex(A)$ again by Lemma~\ref{LEMMA_chrz_extreme_elements}.

To prove (ii), let $p\in \ex(A) \cap B$.
Toward a contradiction,  suppose $\ex(A \cup B) \cap B$ is empty. 
Then, $\ex(A \cup B) \sbs A \cup B$ yields $\ex(A \cup B) \sbs A$, which in turn implies $\conv(A \cup B) \sbs \conv(A)$ by Lemma~\ref{LEMMA_Monjardet_Rad}(ii).  
It follows that $B \sbs \conv(A)$.  
On the other hand,  Lemma~\ref{LEMMA_Monjardet_Rad}(i) gives $p\in \ex(\conv(A))$.  
By Lemma~\ref{LEMMA_extreme_for_convex}(i), it follows that $\conv(A)\setminus\{p\}$ is a convex set.  
Since the latter set includes $(A \cup B) \setminus\{p\}$ but does not contain $p$, Lemma~\ref{LEMMA_chrz_extreme_elements} entails $p \in \ex(A \cup B)$. 
We conclude that $p \in \ex(A \cup B) \cap B = \es$, a contradiction.
\end{proof}

Lemma~\ref{LEMMA_alpha_and_rho_for_CG}(i) has a simple rephrasing that we will often use: \textsl{If an element of a set $A$ is not extreme in $A$, then it cannot be extreme in any superset of $A$.}

\begin{remark}
Lemma~\ref{LEMMA_alpha_and_rho_for_CG}(ii) cannot be strengthened (as its proof might suggest) by requiring that the inclusion $\ex(A) \cap B \subseteq \ex(A \cup B) \cap B$ holds.
Consider, for instance, the convex geometry $\GGG_5$ on $X = \{x,y,z\}$ defined in Example~\ref{EX_all_CGs_on_3_elements_up_to_iso}.
Then, for $A :=\{x,z\}$ and $B := X$, we have  
$$
\ex_{\GGG_5}(A) \cap B = \ex_{\GGG_5}(A) = \{x,z\} \nsubseteq \{y,z\} = \ex_{\GGG_5}(B) = 
\ex_{\GGG_5}(A \cup B) \cap B\,. 
$$	
\end{remark}

The next definition, due to \citet{Plott1973}, introduces `path independent choice spaces', which are important structures in mathematical economics.  

\begin{definition} \rm \label{DEF_choice spaces}
	A function $c \colon 2^X \to 2^X$ such that, for any $A \in 2^X$,
\begin{enumerate}[(i)]
	\item $c(A) \sbs A$, and
	\item $A\neq \es \; \implies \; c(A)\neq \es$ 
\end{enumerate}
is called a \textsl{choice correspondence}; in this case, the pair $(X,c)$ is a \textsl{choice space}.
If, in addition, $c$ satisfies the property of \textsl{path independence}, i.e.,  
\begin{enumerate} 
	\item[\rm \quad(iii)] $c(A \cup B) \;=\; c(c(A) \cup c(B))$
\end{enumerate}
for any $A,B \in 2^X$, then $(X,c)$ is a \textsl{path independent} choice space. 
\end{definition}
 
Path independent choice spaces come up in the analysis of convex geometries because of the following striking result:

\begin{theorem}[\citealp{Koshevoy1999}] \label{THM_intriguing}
If $(X,\GGG)$ is a convex geometry, then $(X,\ex)$ is a path independent choice space; in particular, for all $A,B \in 2^X$, we have 
\begin{equation*} 
	\ex(A \cup B) \;=\; \ex(\ex(A) \cup \ex(B)).
\end{equation*}
Conversely, if $(X,c)$ is a path independent choice space, then there is a unique convex geometry $\GGG_c$ on $X$ whose extreme operator  coincides with $c$, namely  
$$
\GGG_c = \left\{G \in 2^X \:\st\: \left(\forall A \in 2^X\right) \: \big(c(A) = c(G) \:\Longrightarrow\: A \sbs G\big) \right\}.
$$  
\end{theorem} 

Theorem~\ref{THM_intriguing} allows one to establish properties of convex geometries by translating properties of path independent choice spaces.  
For instance, the two implications in Lemma~\ref{LEMMA_alpha_and_rho_for_CG} can also be derived from path independence: see, e.g.,~\citet{Moulin1985} for (i), and~\citet{Cantone_Giarlotta_Watson2020+} for some rephrasing of (ii).
Several additional properties of the extreme operator of a convex geometry $(X,\GGG)$ can also be derived as analogous features of choice correspondences.
 For instance, the following property (\textsl{Aizerman Property}) holds for  any $A,B \in 2^X$:  
\begin{equation*} \label{EQ_Aizerman}
	\ex(B) \subseteq A \subseteq B \; \implies \; \ex(A) \subseteq \ex(B)\,.
\end{equation*}
In fact, \cite{Aizerman_Malishevski1981} show that the join of the property in Lemma~\ref{LEMMA_alpha_and_rho_for_CG}(i)\footnote{This property is originally due to~\cite{Chernoff1954}.} and Aizerman Property is equivalent to path independence. 

We close this section with two remarks about the main examples of convex geometries that will be examined in this paper.

\begin{remark} \label{REM_extreme_operator_for_ordinal_CG_is_maximization}
 Let $(X,\GGG)$ be an ordinal convex geometry having $\le$ as associated partial order.
It is immediate to check that the extreme operator coincides with the operator of maximization w.r.t.\ $\le$, that is, for all nonempty $A \sbs X$,  
$$
\ex(A) = \max(A,\leq) := \left\{a \in A \:\st\: (\forall a' \in A) \; a \le a' \Longrightarrow a=a'\right\}.
$$
In other words, the extreme elements of a set in an ordinal convex geometry are exactly the non-dominated ones.  
 The choice space $(X,\ex)$ that arises in view of Theorem~\ref{THM_intriguing} is said to be `rationalizable' by the partial order $\le$.\footnote{\label{FOOTNOTE:rationalizability} The rationalizability of a choice space is a well-known notion within the classical \textsl{theory of revealed preferences}, pioneered by the economist~\cite{Samuelson1938}. 
 Formally, a choice space $(X,c)$ is \textsl{rationalizable} if there is an acyclic binary relation $\precsim$ (not necessarily a partial order) on $X$ such that, for any $A \in 2^X$, the equality $c(A) = \max(A,\precsim)=\{a \in A \st (\nexists a' \in A) \, a \prec a'\}$ holds.  We refer the reader to the collection of papers by \cite{Suzumura2016} for a vast account of the topic of rationalizable choices.  For some related notions of \textsl{bounded rationality}, which use binary relations or similar tools to explain choice behavior in different ways, see the recent survey  by~\cite{Giarlotta2019} and references therein.}   
\end{remark} 

\begin{remark}
For a finite subset $X$ of a real affine space $\R^d$ as in Example~\ref{EX_affine_CG}, let  $\GGG$ be the convex geometry induced on $X$.
Then the extreme elements of a subset $A$ of $X$ are the vertices of the polytope 
$\conv_\R(A)$.
In other words, $\ex(A)$ is composed of all the points $a$ of $A$ with the property that some linear functional on $\R^d$ is maximized on $A$ at $a$ but at no other point of $A$.
\end{remark}


\section{Resolutions of Convex Geometries} \label{SECT:resolutions cg}

Here we introduce the main notion of this paper, namely resolutions of convex geometries.
This is an application of a general notion of resolution, which captures the concept of expanding a mathematical structure by substituting each of its elements by structures of the same type.  

Historically, resolutions were originally defined in a topological setting by \cite{Fedorchuk1968}, and then extensively studied by \cite{Watson1992}.
Resolutions have proven to be extremely useful in set-theoretic topology, providing a unified point of view for many seemingly different topological spaces.

Very recently, the notion of resolution has been adapted to the field of choice theory by \cite{Cantone_Giarlotta_Watson2020+}.  
Upon restricting this notion to path independent choice spaces, and in view of the structural bijection provided by Theorem~\ref{THM_intriguing}, resolutions of convex geometries naturally arise.


\subsection{Definition and Examples} \label{SUBSECT_preliminary for resolutions}

\begin{definition}\label{DEF_resolution_of_cg} 
Let $(X,\GGG_X)$ be a convex geometry, the \textsl{base geometry}, and $\{(Y_x,\GGG_x) \st x\in X\}$ a family of convex geometries, the \textsl{fiber geometries}, where the sets $Y_x$'s are pairwise disjoint and also disjoint from $X$.  
Let 
$$
Z := \bigcup_{x \in X} Y_x\,,
$$
and define the \textsl{projection} by
$$
\pi \colon Z \to X\,,\quad z \mapsto x \hbox{~ for all } x \in X \hbox{ and } z \in Y_x. 
$$    
The \textsl{resolution of $(X,\GGG_X)$ into $\{(Y_x,\GGG_x) \:\st \: x\in X\}$} is the pair $(Z,\GGG_Z)$, where $\GGG_Z$ consists of all subsets $A$ of $Z$ satisfying the following three requirements:
\begin{enumerate}[\sf (R1)]
\item \textsl{(base coherence)\;} $\pi(A) \in \GGG_X\,$;\label{(R1)}

\item \textsl{(fiber coherence)\;} $A \cap Y_x \in \GGG_x$ for all $x \in \pi(A)\,$;
\label{(R2)}  
\item \textsl{(non-extreme indiscernibility)\;} $Y_x \sbs A$ for all $x \in \pi(A) \setminus \ex_{\GGG_X}(\pi(A))\,$.
\label{(R3)}
\end{enumerate} 
We shall use the suggestive notation
\begin{equation*} 
	(Z,\GGG_Z) = (X,\GGG_X) \boxleft (Y_x,\GGG_x)_{x \in X}\,,
\end{equation*} 
and generically call $(Z,\GGG_Z)$ \textsl{a resolution of convex geometries}. 
A resolution of convex geometries is \textsl{nontrivial} if both the base and at least one fiber have more than one element.  (See Figure~\ref{FIG:definition_of_resolutions}, where each element $x$ of the base $X=\{1,2,3,4,5\}$ is \textsl{resolved} into a fiber $Y_x$, to obtain a nontrivial resolution.) 

\smallskip
Similarly, we define the \textsl{composition}  
\begin{equation*} 
	(Z,\CCC_Z) = (X,\GGG_X) \boxminus (Y_x,\GGG_x)_{x \in X}
\end{equation*}
\textsl{of the base} $(X,\GGG_X)$ \textsl{into the fibers} $(Y_x,\GGG_x)$, $x\in X$, by letting $\CCC_Z$ be the family of all subsets of $Z$ that satisfy \textsf{(R1)} and \textsf{(R2)} (but not necessarily \textsf{(R3)}).
The nontriviality of a composition of convex geometries is defined similarly to resolutions.  
\end{definition}
 

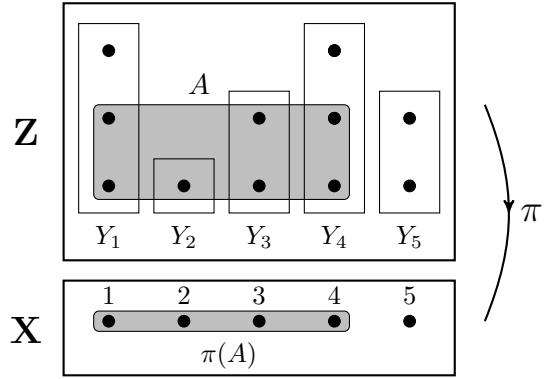
\begin{figure}[ht]
\begin{center}
\begin{tikzpicture}[yscale=0.9]
\tikzstyle{vtx}=[circle,draw,fill=black, scale=0.4]

\draw[draw=black,fill=lightgray,rounded corners=2pt] (-0.2,-1.2) rectangle (3.2,0.2);
\node at (1.2,0.5) {\small $A$}; 

\draw[draw=black,fill=lightgray,rounded corners=2pt] (-0.2,-3.15) rectangle (3.2,-2.85); 
\node at (1.6,-3.5) {\footnotesize $\pi(A)$};

\node[vtx,label=above:\footnotesize$1$] (x) at (0,-3){};
\node[vtx,label=above:\footnotesize$2$] (y) at (1,-3){}; 
\node[vtx,label=above:\footnotesize$3$] (z) at (2,-3){};
\node[vtx,label=above:\footnotesize$4$] (t) at (3,-3){};
\node[vtx,label=above:\footnotesize$5$] (v) at (4,-3){};
\draw[thick] (-0.6,-3.8) rectangle (4.6,-2.4);
\node at (-1.1,-3.13) {\large $\mathbf{X}$};

\node[vtx] (x1) at (0,1){};
\node[vtx] (x2) at (0,0){};
\node[vtx] (x3) at (0,-1){};
\draw (-0.4,-1.4) rectangle (0.4,1.4);
\node at (0,-1.75) {\footnotesize $Y_{1}$};
\node at (1,-1.75) {\footnotesize $Y_{2}$};
\node at (2,-1.75) {\footnotesize $Y_{3}$};
\node at (3,-1.75) {\footnotesize $Y_{4}$};
\node at (4,-1.75) {\footnotesize $Y_{5}$};

\node[vtx] (y1) at (1,-1){};
\draw (0.6,-1.4) rectangle (1.4,-0.6);

\node[vtx] (z1) at (2,-1){};
\node[vtx] (z2) at (2,0){};
\draw (1.6,-1.4) rectangle (2.4,0.4);

\node[vtx] (u1) at (3,1){};
\node[vtx] (u1) at (3,1){};
\node[vtx] (u2) at (3,0){};
\node[vtx] (u2) at (3,0){};
\node[vtx] (u3) at (3,-1){};
\draw (2.6,-1.4)  rectangle (3.4,1.4);

\node[vtx] (v1) at (4,-1){};
\node[vtx] (v2) at (4,0){};
\draw (3.6,-1.4)  rectangle (4.4,0.4); 

\draw[thick] (-0.6,-2.1) rectangle (4.6,1.7);
\node at (-1.1,-0.2) {\large $\mathbf{Z}$};
 
\draw[thick,->-=.5,bend left=20] (5,0.2) to node[right]{\large $\pi$} 
(5,-3);
\end{tikzpicture}

\kern-6mm
\end{center}
\caption{A nontrivial resolution $(Z,\GGG_Z) = (X,\GGG_X) \boxleft (Y_x,\GGG_x)_{x \in X}$, having $X = \{1, \ldots, 5\}$ as base set; each fiber $Y_x$ is above the corresponding element $x \in X$.  The set $A \sbs Z$ and its projection $\pi(A) = \{1,2,3,4\}$ are emphasized in grey. \label{FIG:definition_of_resolutions}}
\end{figure}

Compositions of convex geometries are special cases of `compound of hypergraphs' \citep{Chein_Habib_Maurer1981}, and of (the equivalent notion of) `compositions of set systems' \citep{Mohring_Radermacher1984}, hence also of the more abstract `set operads' \citep{Mendez2015}.
Contrary to what happens for resolutions, compositions do not differentiate fibers according to the status of their index (whereas resolutions do because of Requirement~\textsf{(R3)}). 
This is the reason why a composition of convex geometries may fail to be a convex geometry, as the next example shows. 

\begin{example}\label{ex_compos_not_convex_geometry}
Let  
\begin{alignat*}{2}
	X\;&=\;\{1,2\},\qquad& \GGG_X\;&=\;\{\es, \{1\}, X\},\\
	Y_1\;&=\;\{a_1,b_1\},\qquad& \GGG_1\;&=\; 2^{Y_1},\\
	Y_2\;&=\;\{a_2\},\qquad& \GGG_2\;&=\; 2^{Y_2}.
\end{alignat*}
Clearly, $(X,\GGG_X)$ and $(Y_i,\GGG_i)$, for $i=1,2$, are all convex geometries; in fact, they are even ordinal. 
Their (nontrivial) composition $\CCC_Z$, defined on $Z=\{a_1,b_1,a_2\}$, contains both $\{a_1,a_2\}$ and $\{b_1,a_2\}$ but not their intersection $\{a_2\}$, and so it fails to be a convex geometry.  
Instead, the resolution $(Z,\GGG_Z) = (X,\GGG_X) \boxleft (Y_i,\GGG_i)_{i =1}^2$ is a convex geometry (in fact, an ordinal one), because $\GGG_Z =\{\es, \{a_1\}, \{b_1\}, \{a_1,b_1\},  Z\}$.
\end{example}

As announced in the Introduction, we have:

\begin{theorem}\label{THM_resolution}
A resolution of convex geometries is a convex geometry.
\end{theorem}

\begin{proof}
Let $(X,\GGG_X)$ be a convex geometry, and $\{(Y_x,\GGG_x) \st x \in X\}$ a family of convex geometries such that the sets $Y_x$'s are pairwise disjoint and disjoint from $X$.
To prove that also $(Z,\GGG_Z)= (X,\GGG_X) \boxleft (Y_x,\GGG_x)_{x \in X}$ is a convex geometry, we show that $\GGG_Z$ satisfies Axioms~\textsf{(G1)--(G3)} in Definition~\ref{DEF_convex_geometry}  by using Requirements~\textsf{(R1)--(R3)} in Definition~\ref{DEF_resolution_of_cg}. 
To simplify notation, here we abbreviate $\ex_{\GGG_X}$ into $\ex$. 

\begin{itemize}
  \item[\textsf{(G1)}] The empty set is in $\GGG_Z$, because it trivially satisfies Requirements~\textsf{(R1)--(R3)}. 
  \item[\textsf{(G2)}] To prove that $\GGG_Z$ is closed under intersection, suppose $F$ and $G$ are arbitrary elements of $\GGG_Z$, and so they satisfy \textsf{(R1)--(R3)}.  We  show that $F \cap G$ satisfies \textsf{(R1)--(R3)} as well.  
  
  For \textsf{(R1)}, we prove $\pi(F \cap G) \in \GGG_X$. 
  Set $E := (\pi(F) \cap \pi(G)) \setminus \pi(F \cap G)$.
  Elementary computations yield $\pi(F \cap G) = (\pi(F) \setminus E) \cap (\pi(G) \setminus E)$.
  Since $\GGG_X$ is closed under intersection, it suffices to show that both $\pi(F) \setminus E$ and $\pi(G) \setminus E$ are in $\GGG_X$. 
  
  \textsc{Claim:} $E \sbs \ex(\pi(F)) \cap \ex(\pi(G))$. 
  Indeed, if $x \in E$, then $F \cap Y_x \neq \es$, $G \cap Y_x \neq \es$, and $(F \cap Y_x) \cap (G \cap Y_x) = \es$.   
  Then $x\in\pi(F)$ and not $Y_x \sbs F$, so by \textsf{(R3)} $x\in\ex(\pi(F))$.  
  The proof that the inclusion $E \sbs \ex(\pi(G))$ holds is similar,  and  so the Claim is established.
    
   Since $\pi(F),\pi(G) \in \GGG_X$ by \textsf{(R1)}, now the Claim and Lemma~\ref{LEMMA_extreme_for_convex}(iii) allow us to conclude what we were after, namely $\pi(F) \setminus E$ and $\pi(G) \setminus E$ are in $\GGG_X$. 
   
Proving \textsf{(R2)} is easy: we know $F \cap Y_x \in \GGG_x$
and $G \cap Y_x \in \GGG_x$, which implies $(F \cap G) \cap Y_x=(F \cap Y_x) \cap (G \cap Y_x)$ is in $\GGG_x$, as claimed. 

Finally, to establish \textsf{(R3)}, let $x \in \pi(F \cap G) \setminus \ex(\pi(F \cap G))$. 
Then $x \in \pi(F)$, and moreover $x \notin \ex(\pi(F))$ by Lemma~\ref{LEMMA_alpha_and_rho_for_CG}(i) (for $A=\pi(F \cap G)$ and $B=\pi(F)$).  
Thus, since $F$ satisfies \textsf{(R3)}, we have $Y_x \sbs F$.  
Similarly, $Y_x \sbs G$.  
It follows that $Y_x \sbs F \cap G$, proving that $F \cap G$ satisfies \textsf{(R3)}.

\item[\textsf{(G3)}] To show that $\GGG_Z$ is upgradable, let $F\in\GGG_Z\setminus\{Z\}$.  
If $F$ hits some fiber $Y_x$ without including all of it, then $F \cap Y_x$ belongs to $\GGG_x$ by \textsf{(R2)},
hence, by the upgradability of the convex geometry $\GGG_x$,  there is some $y$ in $Y_x \setminus F$ such that $(F \cap Y_x) \cup \{y\} \in \GGG_x$. 
It follows that $F \cup \{y\}$ is in $\GGG_Z$.
On the other hand, if $F$ is a union of fibers, then, since $F \neq Z$ and $\pi(F)\in\GGG_X$, there is $x$ in $X \setminus \pi(F)$ such that $\pi(F) \cup \{x\} \in \GGG_X$.  
There exists $y$ in $Y_x$ such that $\{y\}\in\GGG_x$.  Then $F \cup \{y\} \in \GGG_Z$ (because $x\in\ex(\pi(F\cup\{y\})$).
\end{itemize}
This completes the proof.  
\end{proof}

\begin{remark}  \label{REM_trivia_on_resolutions_cg} 
Let $(Z,\GGG_Z)= (X,\GGG_X) \boxleft (Y_x,\GGG_x)_{x \in X}$ be a resolution of convex geometries. 
For any $A \in 2^X$, we have
\begin{equation} \label{EQ_equivalence0}
\pi^{-1}(A) \in \GGG_Z \quad \iff \quad A \in \GGG_X ,
\end{equation}
hence, in particular, for any $x \in X$, 
\begin{equation} \label{EQ_equivalences_for_cg}
\GGG_x \sbs \GGG_Z \quad\iff\quad Y_x \in \GGG_Z \quad\iff\quad
\{x\} \in \GGG_X.
\end{equation}
(To prove that $Y_x \in \GGG_Z$ implies $\GGG_x \sbs \GGG_Z$, observe that Requirement \textsf{(R3)} is vacuously satisfied, because the unique element of $\{x\}$ is extreme.)
Also, the resolution is atomistic if and only if the base and all fibers are atomistic.
Moreover, for any $x \in X$, the equivalence
\begin{equation} \label{EQ_equivalence2}
Z \setminus Y_x \in \GGG_Z \quad \iff \quad x \in \ex_{\GGG_X}(X)
\end{equation}
holds. 
Any fiber $(Y_x,\GGG_x)$ is a subgeometry of $(Z,\GGG_Z)$.
Furthermore, if $S \sbs Z$ is a \textsl{transversal} (i.e., $S$ intersects each fiber in exactly one element), then the convex geometry induced by $(Z,\GGG_Z)$ on $S$ is isomorphic to $(X,\GGG_X)$. 
\end{remark}

We shall show in Section~\ref{SECT:special types of cg} that the two main classes of convex geometries examined in this paper, namely ordinal (Example~\ref{EX_ordinal_Convex_Geometries}) and affine (Example~\ref{EX_affine_CG}), behave in a radically different way with respect to resolutions: in fact, ordinality is preserved whereas affineness is not. 
Specifically, the resolution of ordinal convex geometries is again ordinal, and this procedure encapsulates a well-known construction on posets (Theorem~\ref{THM_composition_of_partial_orders}).
On the other hand, we shall provide an elementary example of a resolution of affine convex geometries that fails to be affine (Example~\ref{EX_affinity_is_unstable_under_resolutions}).

We conclude these preliminaries with a technical result, which describes the extreme operator and the convex hull operator in a resolution.

\begin{lemma}\label{LEMMA_conv_ex} 
Let $(Z,\GGG_Z) = (X,\GGG_X) \boxleft (Y_x,\GGG_x)_{x \in X}$ be a resolution of convex geometries. 
For any $A \subseteq Z$, we have
\begin{equation}\label{EQ_conv_Z}
	\conv_{\GGG_Z}(A) \;=\; \bigcup_{x \,\in\, \ex_{\GGG_X}\!(\pi(A))}  \conv_{\GGG_x}(A \cap Y_x)   \;\cup\;  \bigcup_{x \,\in \,\conv_{\GGG_X}\!(\pi(A)) \setminus \ex_{\GGG_X}\!(\pi(A))} Y_x
\end{equation} 
and
\begin{equation}\label{EQ_ex_Z}
	\ex_{\GGG_Z}(A) \;=\; \bigcup_{x\, \in \,\ex_{\GGG_X}\!(\pi(A))} \ex_{\GGG_x}(A \cap Y_x).
\end{equation} 
\end{lemma}

\begin{proof} 
Let $A \in 2^Z$.
We prove \eqref{EQ_conv_Z}.
To start, we show that the right-hand side of \eqref{EQ_conv_Z} belongs to $\GGG_Z$. 
Indeed, it satisfies \textsf{(R1)}, because its projection is equal to $\conv_{\GGG_X}(\pi(A))$, and so it is in $\GGG_X$.
Moreover, it satisfies \textsf{(R2)} and \textsf{(R3)}, too: use the special case $\ex_{\GGG_X}(\conv_{\GGG_X}(\pi(A)))=\ex_{\GGG_X}(\pi(A))$ of Lemma~\ref{LEMMA_Monjardet_Rad}\eqref{item_MR_1}. 
This proves that the right-hand side of \eqref{EQ_conv_Z} is convex.
Therefore, to complete the proof of \eqref{EQ_conv_Z}, it suffices to show that any convex set $G$ in $\GGG_Z$ that includes $A$ also includes the right-hand side.  

To that end, let $G \in \GGG_Z$ be such that $A \sbs G$. 
Take any $x \in \ex_{\GGG_X}(\pi(A))$.  
Since the convex set $G$ includes $A \cap Y_x$, it follows that $G \cap Y_x$ is a convex set in $Y_x$ including $A \cap Y_x$, and so $\conv_{\GGG_x}(A \cap Y_x) \sbs G \cap Y_x \sbs G$. 
On the other hand, let $x \in \conv_{\GGG_X}(\pi(A)) \setminus \ex_{\GGG_X}(\pi(A))$. 
Observe that Requirement~\textsf{(R1)} yields $\pi(A) \sbs \pi(G) \in \GGG_X$, hence $\conv_{\GGG_X}(\pi(A))\sbs \pi(G)$.  Since $\ex_{\GGG_X}(\pi(G)) \cap \conv_{\GGG_X}(\pi(A)) \sbs \ex_{\GGG_X}(\pi(A))$ by Lemma~\ref{LEMMA_alpha_and_rho_for_CG}(i), it follows that $x \in \pi(G) \setminus \ex_{\GGG_X}(\pi(G))$, and so $Y_x \sbs G$ by Requirement \textsf{(R3)}.  
This proves that \eqref{EQ_conv_Z} holds. 
\smallskip

Next, we prove \eqref{EQ_ex_Z}. 
For the forward inclusion, let $z \in \ex_{\GGG_Z}(A)$. 
Set $x := \pi(z)$.
To prove the claim, we show that (i) $x \in \ex_{\GGG_X}(\pi(A))$, and (ii) $z \in \ex_{\GGG_x}(A \cap Y_x)$.
By Lemma~\ref{LEMMA_chrz_extreme_elements}(3), there exists $G \in \GGG_Z$ such that $A\setminus\{z\} \sbs G$ and $z\notin G$.  
Now, if $G \cap Y_x \neq \es$, then $Y_x \nsubseteq G$, and so $x \in \ex_{\GGG_X}(\pi(G))$ by Requirement~\textsf{(R3)}.
It follows that $x \in \ex_{\GGG_X}(\pi(A))$ by Lemma~\ref{LEMMA_alpha_and_rho_for_CG}(i).  
On the other hand, suppose $G \cap Y_x = \es$.
Then the set $\pi(G)$ is convex in $\GGG_X$ and includes $\pi(A)\setminus\{x\}$ but not $x$, which again implies $x \in \ex_{\GGG_X}(\pi(A))$.  
This proves (i).  
For (ii), observe that $G \cap Y_x$ is convex in $Y_x$, includes $(A\cap Y_x)\setminus\{z\}$ but does not contain $z$; therefore, $z \in \ex_x(A \cap Y_x)$.  
This completes the proof of the forward inclusion in \eqref{EQ_ex_Z}.

For the reverse inclusion, let $x \in \ex_{\GGG_X}(\pi(A))$ and $z \in \ex_{\GGG_x}(A \cap Y_x)$.  It suffices to show $z\notin\conv_{\GGG_Z}(A\setminus\{z\})$, which we do by applying  \eqref{EQ_conv_Z} to $A\setminus\{z\}$ in two exhaustive cases. 
If $(A \setminus\{z\}) \cap Y_x = \es$, then $x \notin \conv_{\GGG_x} (\pi(A \setminus \{z\})$ because $x\in \ex_{\GGG_x}(\pi(A))$, and so $z\notin\conv_{\GGG_Z}(A\setminus\{z\})$.
If $(A \setminus\{z\}) \cap Y_x \neq \es$, then $\pi(A\setminus\{z\})=\pi(A)$, so $x\in \ex_{\GGG_x}(\pi(A\setminus\{z\}))$.  Together with $z \notin \conv_{\GGG_x}( (A \setminus\{z\}) \cap Y_x)$, this gives again $z\notin\conv_{\GGG_Z}(A\setminus\{z\})$.
\end{proof}


\subsection{Primitivity vs Resolvability. Shrinkable Sets}
\label{SUBSECT_primitive vs resolvable}

Using resolutions, convex geometries can be partitioned in two classes: 

\begin{definition}\label{DEF_non_trivial}    
 A convex geometry is \textsl{primitive} (or \textsl{irresolvable}) if it cannot be obtained as a nontrivial resolution; otherwise, it is \textsl{resolvable}.  
\end{definition}

How can we \textit{directly} recognize whether a given convex geometry is primitive or resolvable?  
Definition~\ref{DEF_resolution_of_cg} implies that the convex geometry resulting from a resolution is partitioned into fibers.  
Then a naive algorithm to test primitiveness should check one by one all nontrivial partitions of the convex geometry, derive the corresponding base and fibers, and stop when their resolution is the initial convex geometry.  
Obviously, this algorithm is highly ineffective. 
A better algorithm can be derived from Theorem~\ref{THM_characterizing_shrinkable_c} below.
To state it, we need the following notion of `shrinkability': 

\begin{definition}\label{DEF_shrinkable} 
Let $(Z,\GGG)$ be a convex geometry.
A subset $S$ of $Z$ is \textsl{shrinkable} (in $Z$) if it is a nontrivial fiber in some nontrivial resolution producing $(Z,\GGG)$.  
\end{definition}

Observe that if $S$ is shrinkable in $X$, then $1 < \vert S \vert < \vert X \vert$.  
Clearly, a convex geometry is primitive if and only if it has no shrinkable set. 
The following characterization of shrinkability yields a more effective test for the primitiveness of a convex geometry:  

\begin{theorem}\label{THM_characterizing_shrinkable_c} 
Let $(Z,\GGG_Z)$ be a convex geometry.
The following statements are equivalent for a subset $S$ of $Z$ such that $1 < \vert S \vert < \vert Z \vert$: 
\begin{enumerate}[\rm(i)] 
	\item $S$ is shrinkable;
   \item $S$ satisfies the following two properties:
    \begin{enumerate}
    \item[\rm\textsf{(S1)}] $(\forall G \in \GGG_Z)$\; $(G \cap S \neq S \;\implies\; G\setminus S \in \GGG_Z)$;
    \item[\rm\textsf{(S2)}] $(\forall G,H \in \GGG_Z)$\;  
    $(G \cap S \neq \es \land G\setminus S \in \GGG_Z) \;\implies\;(G\setminus S) \cup (H \cap S) \in \GGG_Z$.   
    \end{enumerate} 
\end{enumerate}
\end{theorem}

\begin{proof}
(i)$\implies$(ii). Suppose $S$ is shrinkable. Thus, $(Z,\GGG_Z)$ can be written as  
\begin{equation*}
	(Z,\GGG_Z) = (X,\GGG_X) \boxleft (Y_x,\GGG_x)_{x \in X},
\end{equation*}
where $S = Y_s$ for some $s \in X$.
Let $G,H \in \GGG_Z$, hence $\pi(G),\pi(H) \in \GGG_X$ by \textsf{(R1)}. 
Below we show that Properties~\textsf{(S1)} and \textsf{(S2)} hold for $S=Y_s$. 

For \textsf{(S1)}, suppose $G \in \GGG_Z$ is such that $\es \neq G \cap Y_s \neq Y_s$ 
(if $G \cap Y_s = \es$, then the result holds trivially). 
Then $s \in \ex_{\GGG_X}(\pi(G))$ by \textsf{(R3)}, hence Lemma~\ref{LEMMA_extreme_for_convex}(i) yields $\pi(G \setminus Y_s) = \pi(G) \setminus \{s\} \in \GGG_X$, which proves that \textsf{(R1)} holds for $G \setminus Y_s$. 
Furthermore, $G \setminus Y_s$ satisfies Requirement~\textsf{(R2)}, because so does $G$. 
Thus, to prove that $G \setminus Y_s \in \GGG_Z$, it remains to show that also Requirement~\textsf{(R3)} holds for $G \setminus Y_s$. 
To that end, observe that Lemma~\ref{LEMMA_alpha_and_rho_for_CG}(i) yields 
$\pi(G \setminus Y_s) \cap \ex_{\GGG_X}(\pi(G)) \sbs \ex_{\GGG_X}(\pi(G \setminus Y_s))$.
It follows that any non-extreme element of $\pi(G \setminus Y_s)$ is also non-extreme in $\pi(G)$. 
Hence $G \setminus Y$ satisfies \textsf{(R3)} because so does $G$.
	
For \textsf{(S2)}, suppose $G \in \GGG_Z$ satisfies $G \cap Y_s \neq \es$ and $G \setminus Y_s \in \GGG_Z$. 
If $H \cap Y_s = \es$, then \textsf{(S2)} holds.
Thus, let $H \cap Y_s \neq \es$.
Below we show that $(G\setminus Y_s) \cup (H \cap Y_s)$ is in $\GGG_Z$ by proving that  Requirements~\textsf{(R1)-(R3)} hold for it. 
Requirement~\textsf{(R1)} readily follows from $\pi((G\setminus Y_s) \cup (H \cap Y_s)) = \pi(G) \in \GGG_X$. 
Requirement~\textsf{(R2)} holds, because both $G$ and $H$ satisfy it. 
Finally, for \textsf{(R3)}, observe that $s$ is an extreme element for $\pi(G\setminus Y_s) \cup (H \cap Y_s)) = \pi(G)$, because $\pi(G),\pi(G)\setminus \{s\} \in \GGG_X$. 
Thus, $(G\setminus Y_s) \cup (H \cap Y_s)$ satisfies Requirement~\textsf{(R3)}, because so does $G$. 

\medbreak

(ii)$\implies$(i). Suppose $S$ is a subset of the convex geometry $(Z,\GGG_Z)$ such that $1 < \vert S \vert < \vert Z \vert$, and Properties~\textsf{(S1)} and \textsf{(S2)} hold for all $G,H \in \GGG_Z$.  
To show that $S$ is shrinkable, we shall express $(Z,\GGG_Z)$ as a nontrivial resolution having $S$ as its unique nontrivial fiber.  

Select $s \notin Z$.
Let $\sigma \colon Z \setminus S \to W$ be a bijection from $Z \setminus S$ onto a set $W$ that is disjoint from $Z \cup \{s\}$, and denote $x_z := \sigma(z)$ for all $z \in Z \setminus S$.   
As base set, take $X := W \cup \{s\}$.
Notice that $X \cap Z = \es$ by construction, and $\vert X \vert \geq 2$ because $S \neq Z$ by hypothesis. 
As fiber sets, take $Y_s := S$, and $Y_{x_z} := \{z\}$ for all $x_z \in W = X \setminus \{s\}$. 
Therefore, $Z$ is equal to $\bigcup_{x \in X} Y_x$, where $\vert X \vert \geq 2$ and $\vert Y_s \vert \geq 2$. 
As usual, let $\pi \colon Z \to X$ be the (projection) map defined by $\pi(z) := x_z$ if $z \in Z \setminus S$, and $\pi(z) := s$ if $z \in S$. 
Next, we endow both the base set $X$ and the fiber sets $Y_x$'s with convex geometries induced by $\GGG_Z$.  

Concerning the base, let $\GGG_X$ be the family of subsets of $X$ defined by 
$$
\GGG_X := \{\pi(G) \: \st \: G \in \GGG_Z \}.
$$ 
We now check that $(X,\GGG_X)$ is a convex geometry.
To begin with, observe that $\pi(\es)=\es \in \GGG_X$, that is, Axiom~\textsf{(G1)} holds.

To prove Axiom~\textsf{(G2)} (closure under intersection), we let $\pi(G),\pi(H) \in \GGG_X$, where $G,H \in \GGG_Z$, and show that $\pi(G) \cap \pi(H) \in \GGG_X$. 
We know that $\pi(G \cap H) \sbs \pi(G) \cap \pi(H)$. 
If equality holds, then we are immediately done, since $G \cap H \in \GGG_Z$. 
Otherwise, we must have $\es \neq G \cap S \neq S$, $\es \neq H \cap S \neq S$, $(G \cap H) \cap S = \emptyset$, and $\pi(G) \cap \pi(H) = \pi(G \cap H) \cup \{s\}$.
Now apply Property~\textsf{(S1)} to $G$ and next  Property~\textsf{(S2)} to $G$ and $H$, to derive that the set $K := (G \setminus S) \cup (H \cap S)$ is an element of $\GGG_Z$.
Since $K \cap H \in \GGG_Z$, and $\pi(G) \cap \pi(H) =  \pi(G \cap H) \cup \{s\} = \pi (K \cap H)$, we conclude that $\pi(G) \cap \pi(H)$ is an element of $\GGG_X$ also in this case, as claimed.   
This shows that $\GGG_X$ satisfies Axiom~\textsf{(G2)}.

To prove Axiom~\textsf{(G3)}, let $\pi(G) \neq X$, with $G\in\GGG_Z$.  
By Axiom~\textsf{(G3)} for $(Z,\GGG_Z)$, there is $z_1 \in Z \setminus G$  such that $G_1 := G \cup \{z_1\} \in \GGG_Z$.
If $G \cap S$ is either empty or equal to $S$, there holds $\pi(z_1)\notin \pi(G)$, and the set $\pi(G) \cup \{\pi(z_1)\} = \pi(G_1) \in \GGG_X$ is what we were looking for. 
On the other hand, suppose $\es \neq G \cap S \neq S$.  
If $z_1 \notin S$, then $\pi(G_1) = \pi(G) \cup \{x_{z_1}\}$, and we are done. 
Otherwise, there is $z_2 \in Z \setminus G_1$ such that $G_2 := G_1 \cup \{z_2\} \in \GGG_Z$. 
Again, if $z_2 \notin S$, the claim holds.  
Otherwise, we iterate. 
Since $\pi(G) \neq X$ and $X$ is finite, we shall eventually get $G_n := G_{n-1} \cup \{z_n\} \in \GGG_Z$ such that $\pi(G_i) = \pi(G)$ for all $i = 1,\ldots, n-1$, and $z_n \notin S$. 
Then $\pi(G_n) = \pi(G) \cup \{x_{z_n}\}$ proves that $\GGG_X$ satisfies \textsf{(G3)}.

To complete the definition of the resolution, we provide each fiber $Y_x$ with the convex geometry $\GGG_x$ induced by $\GGG_Z$, that is, 
$$
\GGG_x := \left\{
\begin{array}{lll}
  \{\es, \{x\}\} & \text{ if } x \in W,\\
  \{ G \cap S \:\st \: G \in \GGG_Z \}  & \text{ if } x =s. 
\end{array}
\right.
$$
Below we show that $(Z,\GGG_Z) = (X,\GGG_X) \boxleft (Y_x,\GGG_x)_{x \in X}$, thus completing the proof. 

\smallskip 

To start, we prove that any convex set $G$ from $\GGG_Z$ is also convex in the resolution $(X,\GGG_X) \boxleft (Y_x,\GGG_x)_{x \in X}$, that is, $G$ satisfies Requirements~\textsf{(R1)-(R3)}. 
Clearly $\pi(G)\in\GGG_X$ by the very definition of $\GGG_X$, thus \textsf{(R1)} holds for $G$. 
Furthermore, the intersection of $G$ with an arbitrary fiber $Y_x$ is convex in $Y_x$ by definition of $\GGG_x$, so $G$ also satisfies \textsf{(R2)}.  
Finally, if $G\cap S = \es$ or $G \cap S = S$, then \textsf{(R3)} holds trivially for $G$.  
On the other hand, if $\es \neq G \cap S \neq S$, then we can apply Property~\textsf{(S1)} to $S = Y_s$ to get $G \setminus S \in \GGG_Z$.  
Since $\{s\} = \pi(G) \setminus \pi(G \setminus S)$, we obtain that $s$ is an extreme element in the convex set $\pi(G) \in \GGG_X$.  
We conclude that $G$ satisfies \textsf{(R3)} also in this case, and so $G$ is convex in the resolution $(X,\GGG_X) \boxleft (Y_x,\GGG_x)_{x \in X}$.  

For the reverse inclusion, we show that if $K$ is a nonempty convex set in $(X,\GGG_X) \boxleft (Y_x,\GGG_x)_{x \in X}$, then $K\in\GGG_Z$.  
Since $\pi(K) \in \GGG_X$ by Requirement~\textsf{(R1)} of a resolution, there is $G \in \GGG_Z$ such that $\pi(K)=\pi(G)$.
Clearly, $K \cap Y_x = G \cap Y_x$ for all $x \in W$.
If also $K \cap S = G \cap S$ holds, then we have $K = G$, and we are immediately done.
Since the latter fact trivially holds if $K$ and $S$ are disjoint, we may assume that $K \cap S$ is different from both $\es$ and $G \cap S$ (so $G \cap S$ is nonempty, too).
Next we appeal to Requirement~\textsf{(R2)}, and get  
$K \cap S  \in \GGG_s$, hence, by the very definition of $\GGG_s$, there is $H \in \GGG_Z$ such that $K \cap S = H \cap S \neq \es$. 
It follows that $K = (G \setminus S) \cup (H \cap S)$.
To conclude our proof, we deal separately with the following two exhaustive cases.
 
(1)~In case $S \nsubseteq G$, we have $G \cap S \neq S$, hence Property~\textsf{(S1)} yields $G \setminus S \in \GGG_Z$.  Since $G \cap S \neq \es$, Property \textsf{(S2)} entails $K = (G \setminus S) \cup (H \cap S) \in \GGG_Z$, as required.
 
(2)~In case $S \sbs G$, %
we have $S \nsubseteq K$, because $K \cap S$ differs from $G \cap S$.   
Then, by Requirement~\textsf{(R3)}, $s$ is extreme in $\pi(K)$.  
The latter set is convex in $X$, and moreover it is equal to $\pi(G)$.  
We derive that $\pi(G)\setminus\{s\}$ is also convex in $\GGG_X$, and next that $G \setminus S$ belongs to $\GGG_Z$ (because $\pi(G)\setminus\{s\}$ must be the image by $\pi$ of a convex set from $\GGG_Z$, which must be $G \setminus S$).  
By \textsf{(S2)} we conclude $K = (G \setminus S) \cup (H \cap S) \in \GGG_Z$, as required.
\end{proof}
 
\begin{remark}\label{rmk_G_cap_S_1}
If in \textsf{(S2)} we replace $G \cap S \neq \es$ with $|G \cap S|=1$, then we get an equivalent condition.  
Indeed, assume the modified \textsf{(S2)} is true, and let us prove \textsf{(S2)} in case $|G \cap S| > 1$.   
Since both $G \setminus S$ and $G$ are convex by assumption, there exists some $s \in G \cap S$ such that $G' := (G \setminus S) \cup \{s\} \in \GGG_Z$ (this follows from the axioms of a convex geometry). 
By assumption the modified \textsf{(S2)} holds for $G'$.  
As we have $(G\setminus S) \cup (H \cap S)= (G'\setminus S) \cup (H \cap S)$, it follows that \textsf{(S2)} holds for $G$.
This completes the proof of the equivalence of the two conditions.  
Finally, observe that in the modified \textsf{(S2)}, the element $s$ forming $G \cap S$ is an extreme element of $G$, because $G$ and $G \setminus \{s\}$ are convex.
\end{remark}

Properties~\textsf{(S1)} and \textsf{(S2)} are logically independent, even on some (small) affine convex geometry $(Z,\GGG)$.
The next example shows such independence, 
at the same time exhibiting a resolvable, affine convex geometry. 

\begin {example}\label{ex_two_counter_examples}
Let $Z$ consist of four points $a,b,c,d$ in the real affine plane,  where $c$ belongs to the segment having $b$ and $d$ as extreme points, and $a$ is outside the line passing through the other three  points (see Figure~\ref{FIG_affine_CG_on_4_points}).   
Consider the affine convex geometry $\GGG_Z$ induced on $Z$. 

To show that \textsf{(S1)} and \textsf{(S2)} are mutually independent properties, let $S'=\{a,b\}$ and $S''=\{a,c\}$. 
Then, $S'$ satisfies \textsf{(S1)} (because all subsets of $Z\setminus S'$ are convex) but not \textsf{(S2)} (take $G=\{a,d\}$ and $H=\{b\}$).
On the other hand, $S''$ does not satisfy \textsf{(S1)} (take $G=\{b,c,d\}$) but satisfies \textsf{(S2)} (because there are only two sets in $Z$ that are not convex, namely  $\{b,d\}$ and $\{a,b,d\}$).

Observe also that $S=\{b,c,d\}$ is a shrinkable set in $(Z,\GGG_Z)$, which is therefore a resolvable convex geometry. 
In fact, it is easy to check that $(Z,\GGG_Z)$ is the resolution with base $(\{1,2\},2^{\{1,2\}})$ and fibers $(Y_1,\GGG_1)$ and $(Y_2,\GGG_2)$, where
\begin{alignat*}{2}
  Y_1\;&=\;\{a\},\qquad& \GGG_1\;&=\; 2^{Y_1},\\
  Y_2\;&=\;\{b,c,d\},\qquad& \GGG_2\;&=\; 2^{Y_2} \setminus \{b,d\}.
\end{alignat*}
Here, the base, the fibers, and the resolution are affine convex geometries.
\end{example}


\begin{figure}[ht]
\begin{center}
\begin{tikzpicture}[yscale=0.75]
\tikzstyle{point}=[circle,draw,fill=black, scale=0.25]

\node[point,label=right:\footnotesize$b$] (b) at (1,1){};
\node[point,label=right:\footnotesize$c$] (c) at (1,0){}; 
\node[point,label=right:\footnotesize$d$] (d) at (1,-1){};
\node[point,label=left:\footnotesize$a$] (a) at (-1,0){};
\node at (2,0.8) {$S$};
\draw (c) ellipse (7mm and 14mm);
\draw (1,-1.5)--(1,1.5);
\end{tikzpicture}

\kern-6mm
\end{center}
\caption{An affine convex geometry on four points, having $S$ as shrinkable set (see Example~\ref{ex_two_counter_examples}). \label{FIG_affine_CG_on_4_points}}
\end{figure}
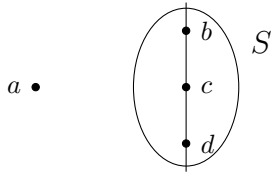

For compositions (or compounds) of hypergraphs, `committees' play a role akin to our shrinkable sets.  
According to \citet{Chein_Habib_Maurer1981}, a \textsl{committee} in a hypergraph $(X,\GGG)$  (where by definition $\GGG \sbs 2^X$) is any subset $S$ of $X$ satisfying a property similar to \textsf{(S2)} in Theorem~\ref{THM_characterizing_shrinkable_c}, where the antecedent is $G \cap S \neq \es$ and $H \cap S \neq \es$.

\cite{Cantone_Giarlotta_Watson2020+} characterize shrinkable sets in the more general setting of choice spaces.
In view of Koshevoy's result (Theorem~\ref{THM_intriguing}) and a simplification due to the path independence of the associated choice space, this directly yields an alternative test for the primitiveness of a convex geometry which employs the extreme operator.

\begin{theorem}\label{THM_characterizing_shrinkable_CGW} 
Let $(Z,\GGG)$ be a convex geometry.
The following statements are equivalent for a subset $S$ of $Z$ such that $1 < \vert S \vert < \vert Z \vert$:
\begin{enumerate}[\rm(i)] 
\item $S$ is shrinkable;
\item $S$ satisfies the following properties for any $A \in 2^Z$: 
    \begin{enumerate}
      \item[\T1]  $\ex_\GGG(A) \cap S \neq \es \; \implies \; \ex_\GGG(A \cap S) \sbs \ex_\GGG(A)\,$;
      \item[\T2] $A \cap S \neq \es \; \implies \; \ex_\GGG(A) \setminus S \sbs \ex_\GGG(A \cup S)\,$;
      \item[\rm\T3] $\big(A \cap S \neq \es \; \wedge \; \ex_\GGG(A \cup S) \cap S \neq \es\big) \; \implies \; \ex_\GGG(A) \cap S \neq \es\,$. 
    \end{enumerate} 
\end{enumerate}  
Moreover, when $(Z,\GGG)$ is atomistic, \T1 implies \T3, and so Properties \T1 and \T2 characterize the shrinkability of $S$.
\end{theorem}

We now establish the independence among the three Properties~\T1--\T3.

\begin {example}\label{ex_two_other_counter_examples}
As in Example~\ref{ex_two_counter_examples}, consider the affine convex geometry induced on four points $a$, $b$, $c$ and $d$ in the real affine plane, with $c$ between $b$ and $d$ and moreover $a$ outside the line through $b$, $c$ and $d$. 
If we let $S=\{a,b,c\}$, then $S$ does not satisfy \T1 (for $A=\{a,b,c\}$) but it satisfies \T2 and \T3.
If we now let $S=\{a,d\}$, then $S$ satisfies \T1 and thus also \T3,  but not \T2 (for $A=\{a,b,c\}$).
Thus even in affine convex geometries, \T1 and \T2 are each one independent of the other two properties.

To show that \T3 is independent of \T1 and \T2, consider the ordinal convex geometry derived from the partial order $\le$ on $Z=\{a,b,c\}$ with $a < b$, $a < c$ and no other strict comparison. Then the subset $S=\{a,b\}$ satisfies \T1 and \T2, but not \T3 (take $A=\{a,c\}$).
\end{example}


\subsection{Extreme Resolutions and Extremely Shrinkable Sets}
\label{SUBSECT_extreme_resolutions}

Here we study resolutions whose nontrivial fibers are all indexed by extreme elements of the base. 

\begin{definition}\label{DEF_extreme_resolution} 
A resolution
$$
(Z,\GGG_Z) = (X,\GGG_X) \boxleft (Y_x,\GGG_x)_{x \in X}
$$
is \textsl{extreme} when for each $x\in X$, if $x \notin \ex_{\GGG_X}(X)$ then $|Y_x|=1$. 
\end{definition}
 
As we shall see in Theorem~\ref{THM_resolution_of_convex_geometries}, any resolution of convex geometries that happens to be affine is also extreme.  
The next result uses extremeness to further clarify the link between compositions and resolutions. 

\begin{theorem}\label{THM_extreme_res_composition} 
Let $(Z,\GGG_Z)=(X,\GGG_X) \boxleft (Y_x,\GGG_x)_{x \in X}$ be a  resolution of convex geometries, and   $(Z,\CCC_Z) = (X,\GGG_X) \boxminus (Y_x,\GGG_x)_{x \in X}$ the composition with the same base and fibers.   
Then $\GGG_Z \sbs \CCC_Z$, and $\GGG_Z = \CCC_Z$ holds if and only if the resolution is extreme.
\end{theorem} 

\begin{proof}  
Definition~\ref{DEF_resolution_of_cg} readily yields $\GGG_Z \sbs \CCC_Z$.  
Next we show that if the resolution is extreme, then $\CCC_Z \sbs  \GGG_Z $.  
Indeed, in an extreme resolution, \textit{any} subset $A$ of $Z$ satisfies Requirement~\textsf{(R3)}, because an element $x$ that is non-extreme in $\pi(A)$ is also non-extreme in $X$, hence $Y_x$ has only one element and is contained in $A$. 

Conversely, we show that if the resolution is not extreme, then $\CCC_Z \subseteq \GGG_Z$ does not hold. 
By assumption, there exists some $x$ in $X \setminus \ex_{\GGG_x}(X)$ such that the fiber $Y_x$ contains at least two elements.  By the definition of a convex geometry,  $\GGG_{Y_x}$ contains a nonempty convex set $G$ distinct from $Y_x$.    
Now it is easy to check that $(Z \smu Y_x) \cup G$ belongs to $\CCC_Z$ but not to $\GGG_Z$.
\end{proof}

\begin{remark} \label{REM_reso_of_set_systems}
Theorem~\ref{THM_extreme_res_composition} has a direct extension to a large family of set systems, as we now explain.  
A \textsl{set system} is a pair $(X,\FFF)$, where $X$ is a nonempty set, and $\FFF$ is a nonempty collection of subsets of $X$. 
A set system $(X,\FFF)$ is \textsl{simple} if $\bigcup\FFF=X$; in particular, $(X,\FFF)$ is \textsl{plain} if it is simple and moreover $\FFF\neq\{\es,X\}$ whenever $|X|>1$.  

\textsl{Compositions of set systems} are defined exactly as compositions of convex geometries (Definition~\ref{DEF_resolution_of_cg}).  
To define \textsl{resolutions of set systems}, we only need a notion of extreme element in set systems, and then again copy from Definition~\ref{DEF_resolution_of_cg}.  
Lemma~\ref{LEMMA_chrz_extreme_elements} suggests the following definition: given a set system $(X,\FFF)$, an element $x$ in a subset $A$ of $X$ is \textsl{extreme in} $A$ when there exists some $F \in \FFF$ such that $A \setminus \{x\} \sbs F$ and $x \notin F$ \citep[compare with, for instance,][]{Ando2006}.

Now consider a composition 
$(Z,\FFF_Z)=(X,\FFF_X) \boxminus (Y_x,\FFF_x)_{x \in X}$ and a  resolution  $(Z,\CCC_Z) = (X,\FFF_X) \boxleft (Y_x,\FFF_x)_{x \in X}$ of set systems (with the same base and fibers).  As for convex geometries (Theorem~\ref{THM_extreme_res_composition}), the resolution $\FFF_Z$ is a subcollection of the composition $\CCC_Z$.  
For plain set systems, the arguments in the proof of Theorem~\ref{THM_extreme_res_composition} show that the equality $\FFF_Z=\CCC_Z$ occurs exactly when the following property is satisfied: for any nontrivial fiber $Y_x$ and convex set $F$ in $\FFF_X$ containing $x$, there holds $x \in \ex_{\FFF_X}(F)$.
(Observe that if $X \in \FFF_X$, it suffices to require this property to hold for $F=X$, because an element $x$ that is extreme in $X$ is also extreme in any subset of $X$ containing $x$.) 
\end{remark}

To recognize which convex geometries can be written as a nontrivial extreme resolution, we introduce and characterize a variant of shrinkability.

\begin{definition}\label{DEF_extremely_shrinkable} 
Let $(Z,\GGG_Z)$ be a convex geometry.  
A subset $S$ of $Z$ with at least two elements is \textsl{extremely shrinkable} (in $Z$) if it is a fiber in some nontrivial extreme resolution producing $(Z,\GGG_Z)$. 
Whenever $Z$ contains such a set, we say that $(Z,\GGG_Z)$ is \textsl{extremely resolvable}.
\end{definition}

\begin{theorem}\label{THM_characterizing_extreme_shrinkable} 
Let $(Z,\GGG_Z)$ be a convex geometry.  The following statements are equivalent for a shrinkable set $S \subseteq Z$:  
\begin{enumerate}[\rm(i)] 
\item $S$ is extremely shrinkable; 
\item for any resolution $(X,\GGG_X) \boxleft (Y_x,\GGG_x)_{x \in X}$  equal to $(Z,\GGG_Z)$, if $S$ coincides with a fiber $Y_x$, then $x \in \ex_{\GGG_X}(X)$; 
\item for at least one resolution $(X,\GGG_X) \boxleft (Y_x,\GGG_x)_{x \in X}$ equal to $(Z,\GGG_Z)$, if $S$ coincides with a fiber $Y_x$, then $x \in \ex_{\GGG_X}(X)$;
\item $Z \setminus S \in \GGG_Z$.
\end{enumerate}
\end{theorem}

\begin{proof} 
We prove (i)$\implies$(iv)$\implies$(ii)$\implies$(i), and leave the (simple)  proof of (ii)$\implies$(iii)$\implies$(iv) to the reader. 
Let $T:=Z \setminus S$, where $S$ is shrinkable in $Z$.  

(i)$\implies$(iv).
If $S$ is extremely shrinkable in $(Z,\GGG_Z)$, then by definition there is an extreme resolution $(X,\GGG_X) \boxleft (Y_x,\GGG_x)_{x \in X}$ equal to $(Z,\GGG_Z)$ such that $S=Y_x$ for some $x \in X$.  Since $|S| \ge 2$, the hypothesis yields $x \in \ex_{\GGG_x}(X)$.  We claim that $T \in \GGG_Z$.
Requirement \textsf{(R1)} in Definition~\ref{DEF_resolution_of_cg} holds for $T$ by the equivalence~\eqref{EQ_equivalence2} in Remark~\ref{REM_trivia_on_resolutions_cg}.  Next, as $T \cap Y_x = \es$ and $T \cap Y_{x'} = Y_{x'}$ for $x' \in X\setminus\{x\}$, we derive that \textsf{(R2)} holds for $T$ as well.  
Since $T$ also satisfies \textsf{(R3)}, the implication (i)$\implies$(iv) is fully proved.  

(iv)$\implies$(ii).  Suppose $T \in \GGG_Z$.  In the resolution considered in (ii), the assumption implies $\pi(T) = X \setminus \{x\} \in \GGG_X$.  It follows that $x \in \ex_{\GGG_x}(X)$.

(ii)$\implies$(i).  Suppose (ii) holds.  As $S$ is assumed to be shrinkable, there exists (as in the proof of Theorem~\ref{THM_characterizing_shrinkable_c}) a resolution in which one fiber equals $S$ and all the other fibers have size $1$.  
By (ii), the projection of $S$ is an extreme element of the base.  
As a consequence, the resolution is extreme. 
\end{proof}

A characterization of shrinkable subsets in terms of convex sets  appeared in Theorems~\ref{THM_characterizing_shrinkable_c}.   We derive a simpler characterization for extremely shrinkable sets.

\begin{theorem}\label{THM_characterizing_shrinkable_V_1_2} 
Let $(Z,\GGG_Z)$ be a convex geometry.
The following statements are equivalent for a subset $S$ of $Z$ such that $1 < \vert S \vert < \vert Z \vert$:
\begin{enumerate}[\rm(i)] 
\item $S$ is extremely shrinkable;
\item $S$ satisfies the following properties for any $G,H \in \GGG_Z$:
    \begin{enumerate} 
        \item[\V1] $|G \cap S|=1 \; \implies \; G \cup S \in \GGG_Z$;
        \item[\V2] $(Z\setminus S) \cup (H \cap S) \in \GGG_Z$.  
    \end{enumerate}
\end{enumerate}
\end{theorem}

\begin{proof}
(i)$\implies$(ii).  Assume $S$ is extremely shrinkable.   
To prove \textsf{(V1)}, suppose $G \in \GGG_Z$ and $|G \cap S|=1$.   By Theorem~\ref{THM_characterizing_extreme_shrinkable}, the hypothesis entails $Z \setminus S \in \GGG_Z$, whence $G \setminus S = G \cap (Z \setminus S)\in \GGG_Z$.  
Since $S$ is shrinkable, Property \textsf{(S2)} from Theorem~\ref{THM_characterizing_shrinkable_c} yields  (taking $H:=Z$)  $G \cup S \in \GGG_Z$. 
To prove \textsf{(V2)}, let $H \in \GGG_Z$.  Then, for $G:=Z$ in \textsf{(S2)}, we get $(Z \setminus S) \cup (H \cap S) \in \GGG_Z$, as desired. 

\smallskip

(ii)$\implies$(i).  Assume \textsf{(V1)} and \textsf{(V2)} hold. 
We first derive the shrinkability of $S$ by establishing Properties~\textsf{(S1)} and \textsf{(S2)} in Theorem~\ref{THM_characterizing_shrinkable_c}. 
Property~\textsf{(S1)} holds because $G\setminus S = G \cap (Z \setminus S) \in\GGG_Z$ follows from \textsf{(V2)} with $H = \es$. 
To prove \textsf{(S2)}, let $G,H \in \GGG_Z$ be such that $G \cap S \neq \es$ and $G \setminus S \in \GGG_Z$.  
In view of Remark~\ref{rmk_G_cap_S_1}, we can assume $|G \cap S|=1$.  
Using
$$
(G \setminus S) \cup (H \cap S) \;=\; 
     (G \cup S) \cap \big((Z\setminus S) \cup (H \cap S)\big),
$$
we derive from both \textsf{(V1)} and \textsf{(V2)} that the latter set lies in $\GGG_Z$.  This proves that $S$ is shrinkable. 
Finally, observe that $S$ is also extremely shrinkable, because $Z\setminus S \in \GGG_Z$ follows from \textsf{(V2)} with $H:=\es$, hence we can make use of Theorem~\ref{THM_characterizing_extreme_shrinkable}. 
\end{proof}

\begin{remark}
In case an extremely shrinkable set $S$ is convex in $\GGG_Z$, Property~\textsf{(V2)} becomes equivalent to the following one:
\begin{itemize}
	\item[] \textsf{(V2')\;\;} $\big(H \in \GGG_Z \, \wedge \; H' \sbs S\big) \; \Longrightarrow \; (Z\setminus S) \cup H'\in \GGG_Z$.    
\end{itemize}
The reason is that $\{H \cap S \st H\in\GGG_Z\} = \{H' \in \GGG_Z \st H' \sbs S\}$ whenever $S \in \GGG_Z$.
\end{remark}


\section{Resolutions of Special Convex Geometries} \label{SECT:special types of cg}

Here we examine resolutions of the two classes of convex geometries introduced in Section~\ref{SUBSECT_deff_and_examples}: ordinal and affine.  
We start with the affine case.


\subsection{Affine Convex Geometries}
\label{SUBSECT_affine cg}

Recall from Example~\ref{EX_affine_CG} that a convex geometry is affine if and only if it is isomorphic to a convex geometry induced on a finite subset of a real affine space.

\begin{lemma}\label{LEMMA_affine_are_hered}
If a resolution is an affine convex geometry, then the base and the fibers of the resolution are also affine convex geometries.
\end{lemma}

\begin{proof}
The result follows at once from Remark~\ref{REM_trivia_on_resolutions_cg}: the fibers are subgeometries of the resolution, and the base is isomorphic to a subgeometry of the resolution.
\end{proof}

The converse of Lemma~\ref{LEMMA_affine_are_hered} does not hold:  Examples~\ref{EX_affinity_is_unstable_under_resolutions} and \ref{EX_affinity_is_unstable_under_affinity_is_unstable_under_resolutions_BIS} below show that resolutions of affine convex geometries need not be affine.

\begin{example}\label{EX_affinity_is_unstable_under_resolutions}
	Let
 	$$
 	\begin{array}{lllll}
 		X = \{1,2,3\}\,,& \GGG_X = 2^X \setminus \{\{1,3\}\}\,, \\
 		Y_1 = \{a\} \,,& \GGG_1 = 2^{Y_1}\,, \\
 		Y_2 = \{b,c\} \,,& \GGG_2 = 2^{Y_2}\,, \\
 		Y_3 = \{d\} \,,& \GGG_3 = 2^{Y_3}\,. \\ 
 	\end{array}
 	$$
 	All pairs $(X,\GGG_X)$ and $(Y_i,\GGG_i)$, for $i = 1,2,3$, are affine convex geometries. 
 	A simple computation shows that the resolution of $(X,\GGG_X)$ into $\{(Y_i,\GGG_i) \st i \in X\}$ is the convex geometry $(Z,\GGG_{Z}) = (X,\GGG_{X})\boxleft (Y_{i},\GGG_{i})_{i \in X}$, where $Z = \{a,b,c,d\}$, and 
	$$
	\GGG_Z = 2^Z \setminus \big\{ \{a,d\},\{a,b,d\},\{a,c,d\}\big\}.
	$$
(See Figure~\ref{FIG:non_repres}, which describes the lattice $(\GGG_Z,\subseteq)$ of convex sets.) 
However, the convex geometry $(Z,\GGG_Z)$ is not affine. 
This readily follows from the fact that $b,c\in\conv_{\GGG_Z}(\{a,d\})$, but $b\notin\conv_{\GGG_Z}(\{a,c\})$ and $c\notin\conv_{\GGG_Z}(\{a,b\}$: indeed, if $(Z,\GGG_Z)$ were affinely embedded in some real affine space, we would have $b$ and $c$ on the segment $[a,d]$, and so either $b\in[a,c]$ or $c\in[a,b]$.
\end{example} 
 
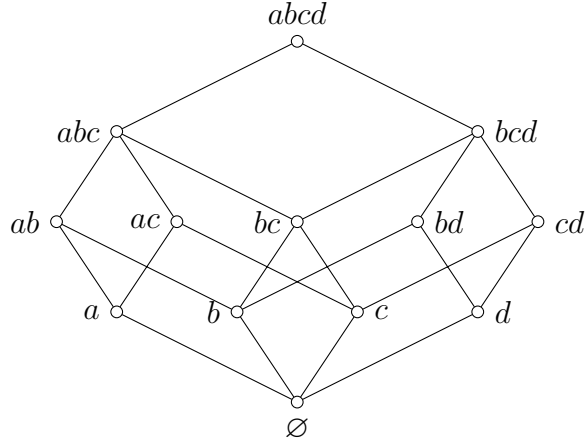
\begin{figure}[ht]
\begin{center}
\begin{tikzpicture}[xscale=0.8,yscale=1.2] 
\tikzstyle{sta}=[circle,draw,fill=black,scale=0.4]
\tikzstyle{irr}=[circle,draw,fill=white,scale=0.4] 
 
\node[irr,label=below:$\es$] (es) at (0,0) {}; 
\node[irr,label=left:$a$] (a) at (-3,1) {};
\node[irr,label=left:$b$] (b) at (-1,1) {};
\node[irr,label=right:$c$] (c) at (1,1) {};
\node[irr,label=right:$d$] (d) at (3,1) {};
        
\node[irr,label=left:$ab$] (ab) at (-4,2) {}; 
\node[irr,label=left:$ac$] (ac) at (-2,2) {};
\node[irr,label=left:$bc$] (bc) at ( 0,2) {};
\node[irr,label=right:$bd$] (bd) at ( 2,2) {};
\node[irr,label=right:$cd$] (cd) at ( 4,2) {};

\node[irr,label=left:$abc$] (abc) at (-3,3) {};
\node[irr,label=right:$bcd$] (bcd) at ( 3,3) {};

\node[irr,label=above:$abcd$] (all) at (0,4) {};  

\draw (es) -- (a) -- (ab) -- (abc) -- (all)
      (es) -- (b) -- (ab)
      (a) -- (ac) -- (abc) 
      (b) -- (bc) -- (abc)
      (b) -- (bd) -- (bcd) -- (all)
      (bc) -- (bcd)
      (es) -- (c) -- (cd) -- (bcd)
      (c) -- (ac)
      (c) -- (bc)
      (es) -- (d) -- (bd)
      (d) -- (cd);
\end{tikzpicture}
\end{center}
\caption{The lattice of convex sets of the (non-affine) resolution in Example~\ref{EX_affinity_is_unstable_under_resolutions}.
\label{FIG:non_repres}}
\end{figure}

Observe that the non-affine resolution $(Z,\GGG_Z)$ of Example~\ref{EX_affinity_is_unstable_under_resolutions} is also non-extreme, since $2 \notin \ex_{\GGG_X}(X)$ and yet $\vert Y_2 \vert >1$.  In fact, to conclude that $(Z,G_Z)$ is non-affine, it suffices to check that it is non-extreme, as the next result guarantees. 

\begin{theorem}\label{THM_resolution_of_convex_geometries}
An affine resolution of convex geometries is extreme. 
\end{theorem}

Theorem~\ref{THM_resolution_of_convex_geometries} is a special case of a more general result (Theorem~\ref{THM_resolution_of_convex_geometries_exchange}), which gives a sufficient condition for a resolution to be extreme. 
This condition is the well-known \textsl{exchange property}, due to \citet{Levi1951} and \citet{Sierksma1984}, applied to a convex geometry $(Z,\GGG_Z)$: for any $A\in 2^Z$ and $p \in Z$,   
\begin{itemize}
	\item[] \!\!\textsf{(EP)} \; $p \in \conv_{\GGG_Z}(A)  \;\implies\; \conv_{\GGG_Z}(A) \;\sbs\; \bigcup_{a\in A}\; \conv_{\GGG_Z}\big((A\setminus\{a\}) \cup \{p\}\big)$.
\end{itemize}  
Theorem~\ref{THM_resolution_of_convex_geometries} follows from Theorem~\ref{THM_resolution_of_convex_geometries_exchange}, because any affine convex geometry satisfies \textsf{(EP)}.

\begin{theorem}\label{THM_resolution_of_convex_geometries_exchange}
A resolution of convex geometries which satisfies the exchange property \textnormal{\textsf{(EP)}} is extreme. 
\end{theorem} 

\begin{proof}
Let $(Z,\GGG_Z)=(X,\GGG_X) \boxleft (Y_x,\GGG_x)_{x \in X}$ be a resolution of convex geometries that satisfies the exchange property. 
Denote by $E$ the set of all extreme points of the base set $X$.    
Toward a contradiction, assume there is $w \in X\setminus E$ such that $\vert Y_{w} \vert \geqslant 2$, and let $a,b$ be two distinct elements of $Y_{w}$. 
Since $w \in X=\conv_{\GGG_X}(E)$, there is a minimal subset $V$ of $E$ such that $w \in  \conv_{\GGG_X}(V)$.  
For each $v \in V$, denote by $V_{v \leftarrow w}$ the set $(V\setminus\{v\}) \cup \{w\}$. 
By the minimality of $V$, we have    
\begin{equation} \label{EQ_help1_for_SC_for_extremality}
	w \in \ex_{\GGG_X} \!\left(V_{v \leftarrow w} \right) \text{ for all } v \in V\,.	
\end{equation}
Now let $T \subseteq Z$ be a transversal for the family $\{Y_v \st v \in V\}$, that is, $\vert T \cap Y_v \vert = 1$ for each $v$ in $V$. 
Observe that $\ex_{\GGG_X}(\pi(T)) = \ex_{\GGG_X}(V) = V$ by Lemma~\ref{LEMMA_extreme_for_convex}(i).
Therefore, Lemma~\ref{LEMMA_conv_ex} yields
$$
	\conv_{\GGG_Z}(T) \;=\; \bigcup_{v \in V}  \conv_{\GGG_v}(T \cap Y_v)   \;\cup\;  \bigcup_{x \,\in \,\conv_{\GGG_X}\!(V) \setminus V} Y_x\,,
$$
and so, since $w \in \conv_{\GGG_X}(V) \setminus V$ and $a,b \in Y_{w}$, we deduce 
\begin{equation} \label{EQ_help3_for_SC_for_extremality}
	a,b \in \conv_{\GGG_Z}(T)\,.	
\end{equation}
Since $a \in \conv_{\GGG_Z}(T)$, the exchange property entails 
$$
\conv_{\GGG_Z}(T) \; \subseteq \; \bigcup_{t \in T} \,\conv_{\GGG_Z} \!\left(T_{t \leftarrow a} \right)\,,
$$ 
where $T_{t \leftarrow a}$ stands for $(T \setminus \{t\}) \cup \{a\}$. 
Since $b \in \conv_{\GGG_Z}(T)$, it follows 
\begin{equation} \label{EQ_help2_for_SC_for_extremality} 
	b \in \conv_{\GGG_Z} \!\left( T_{u \leftarrow a} \right) \text{ for some } u \in T\,.	
\end{equation}
Let $r = \pi(u)$. 
Another application of Lemma~\ref{LEMMA_conv_ex} yields 
\begin{align*}
	\conv_{\GGG_Z} \! \left( T_{u \leftarrow a} \right) = \! \bigcup_{x \, \in \, \ex_{\GGG_X}\!(V_{r \leftarrow w})} \!\!\conv_{\GGG_x} \!\left(T_{u \leftarrow a} \cap Y_x \right) \;  \cup \!
	\bigcup_{x \, \in \, \conv_{\GGG_X}\!(V_{r \leftarrow w}) \setminus  \ex_{\GGG_X}\!(V_{r \leftarrow w})}  \!\!\! Y_x\,,
\end{align*}
whence, by \eqref{EQ_help1_for_SC_for_extremality} and \eqref{EQ_help2_for_SC_for_extremality}, we deduce
$$
b \in \conv_{\GGG_{w}}\!(T_{u \leftarrow a} \cap Y_{w}) = \conv_{\GGG_{w}} \!(\{a\})\,,
$$
which in turn implies $b \notin \ex_{\GGG_{w}}\!(\{a,b\})$.  
The roles of $a$ and $b$ being exchangeable, we also have $a \notin \ex_{\GGG_{w}}\!(\{a,b\})$.
It follows that $\ex_{\GGG_{w}}\!(\{a,b\}) = \es$, a contradiction. 
\end{proof}

\begin{corollary}\label{COR_GminusS}
Let $(Z,\GGG_Z)$ be an affine convex geometry.  For any shrinkable set $S \subseteq Z$, we have:
\begin{enumerate}[\rm (i)]
	\item $S \in \GGG_Z$;
	\item $G \setminus S \in \GGG_Z$ for all $G \in \GGG_Z$, in particular $Z \setminus S \in \GGG_Z$.  
\end{enumerate}
\end{corollary}

\begin{proof} 
Fix $S \subseteq Z$ shrinkable.  
Thus there is a resolution $(X,\GGG_X) \boxleft (Y_x,\GGG_x)_{x \in X}$ producing $(Z,\GGG_Z)$ in which $S$ is a (nontrivial) fiber $Y_x$.
\smallskip

(i) By Lemma~\ref{LEMMA_affine_are_hered}, the base $(X,\GGG_X)$ is an affine convex geometry, which implies that any one-element set $\{x\}$ in $2^X$ belongs to $\GGG_X$.  By \eqref{EQ_equivalences_for_cg} in Remark~\ref{REM_trivia_on_resolutions_cg}, we conclude $S = Y_x \in \GGG_Z$. 

\smallskip

(ii) To start, we prove $Z \setminus S \in \GGG_Z$.  
By  Theorem~\ref{THM_resolution_of_convex_geometries}, the resolution $(Z,\GGG_Z) = (X,\GGG_X) \boxleft (Y_x,\GGG_x)_{x \in X}$ is extreme, and so $S = Y_x$ is extremely shrinkable.
An application of Theorem~\ref{THM_characterizing_extreme_shrinkable} readily yields $Z \setminus S \in \GGG_Z$.  
Then, for arbitrary $G \in \GGG_Z$, we get $G \setminus S \in \GGG_Z$, because $G \setminus S = G \cap (Z \setminus S)$. 
\end{proof}

Finally, we obtain a characterization of all affine convex geometries that are resolvable (hence of primitive ones): 

\begin{corollary}
	The following statements are equivalent for an affine convex geometry $(Z,\GGG_Z)$:
	\begin{enumerate}[\quad \rm (1)]
		\item $(Z,\GGG_Z)$ is resolvable;
		\item there is a shrinkable subset $S$ of $Z$;
		\item there is an extremely shrinkable subset $S$ of $Z$; 
		\item there is a subset $S$ of $Z$, with $1 < \vert S \vert < \vert X \vert$, satisfying the following properties for any $G,H \in \GGG_Z$: 
		\begin{enumerate}[\phantom{zzzz}]
        	\item[\rm\textsf{(V1)}\;\;] $\big(G \in \GGG_Z \; \wedge \; |G \cap S|=1 \big)\; \Longrightarrow \; G \cup S \in \GGG_Z$,
        	\item[\V2\;\;] $H \in \GGG_Z \; \Longrightarrow \;                  (Z\setminus S) \cup (H \cap S) \in \GGG_Z$.  
		\end{enumerate} 
	\end{enumerate}
\end{corollary} 

\begin{proof}
	Simply observe that as a consequence of Theorem~\ref{THM_resolution_of_convex_geometries}, affine resolutions fall under the application of Theorem~\ref{THM_characterizing_shrinkable_V_1_2}. 	
\end{proof}

The primitivity of a convex geometry is characterized by the non-existence of a shrinkable set.  
However, even for affine convex geometries, this characterization is not computationally effective, because   Properties~\textsf{(V1)} and \textsf{(V2)} are to be checked for all convex sets of the given geometry. 
We wonder whether there are more instructive answers to the next two, related problems. 

\begin{problem}\label{PROB_1}
Given an affine convex geometry $(Z,\GGG)$, characterize when a subset of $Z$ is shrinkable.  
\end{problem}

\begin{problem}\label{PROB_2}
Geometrically characterize when an affine convex geometry is primitive.\end{problem}

By Theorems~\ref{THM_resolution_of_convex_geometries} and~\ref{THM_extreme_res_composition}, Problems~\ref{PROB_1} and \ref{PROB_2} are also problems about compositions of affine convex geometries.
Although they appear to be central problems, we were unable to find any mention of them in the literature.

The next result states an equivalent (geometric) formulation of Property~\T1 in  Theorem~\ref{THM_characterizing_shrinkable_CGW}.  
This reformulation is in the spirit of the answers we would like to obtain for Problems~\ref{PROB_1} and \ref{PROB_2}.
In what follows, `$\conv_{\GGG_Z}$' denotes the convex hull operator in a convex geometry $\GGG_Z$, whereas `$\conv_\R$' is used for the standard convex hull in the affine space $\R^d$.

\begin{proposition} \label{PROP_chrz_tracing_in_affine_geometry}
Let $Z$ be a finite subset of $\R^d$, and $(Z,\GGG)$ the convex geometry induced on $Z$.
The following statements are equivalent for any set $S \sbs Z$ such that $1 < |S| < |Z|\,$:
\begin{enumerate}[\rm (i)]
	\item $S$ satisfies Property~\T1 in Theorem~\ref{THM_characterizing_shrinkable_CGW}, namely for all $A\in2^Z$,  
	  \begin{enumerate}
      \item[\T1] $\ex(A) \cap S \neq \es \;\; \Longrightarrow \;\; \ex(A \cap S) \sbs \ex(A)$; 
      \end{enumerate}
	\item there exist proper faces $F_1, F_2, \dots, F_k$ ($k \ge 1)$ of the convex polytope $\conv_\R(Z)$ such that  $S = Z \cap (F_1 \cup F_2\cup \dots \cup F_k)$.
\end{enumerate}
\end{proposition}

A simple consequence of Condition~(ii) is that $S$ lies in the relative boundary\footnote{That is, the boundary computed in the affine subspace generated by $\conv_\R(Z)$.} of $\conv_\R(Z)$. 
However, Condition (ii) asserts more than that.

\begin{proof}
Fix $S \sbs Z$ such that $1 < \vert S \vert < \vert Z\vert$. 

\medskip

(ii) $\Longrightarrow$ (i): 
Suppose there exist proper faces $F_1$, $F_2$, \dots, $F_k$ of the convex polytope $\conv_\R(Z)$ such that  $S = Z \cap (F_1 \cup F_2\cup \dots \cup F_k)$.  
If any face $F_i$ equals $\conv_\R(Z)$, then $S=Z$, and so $S$ satisfies \T1.  
Thus we may assume that all $F_i$'s are proper faces of $\conv_\R(Z)$.  
Let $A \in 2^Z$; we shall show that $\ex(A \cap S) \sbs \ex(A)$.  
Given $w$ in $\ex(A \cap S)$, we know by (ii) that $w$ belongs to some face $F_i$ of $\conv_\R(Z)$, with moreover $Z \cap F_i \sbs S$.  
If $w \in \ex(A)$ does not hold, then there exists a subset $B$ of $A\setminus\{w\}$ such that $w \in \conv_\R(B)$.  
Such a minimal subset $B$ of $A\setminus\{w\}$ is formed by the vertices $b_1$, $b_2$, \dots, $b_\ell$ of a simplex containing $w$ in its relative interior.  
Then all $b_j$'s belong to $F_i$, because the proper face $F_i$ of $\conv_\R$ equals the intersection of $\conv_\R(Z)$ with some hyperplane supporting $\conv_\R(Z)$.  It follows that, for all $j$'s,  we have $b_j \in A \cap F_i \sbs A \cap Z \cap F_i \sbs A \cap S$, contradicting the initial assumption $w \in \ex(A \cap S)$.  

\medskip

(i) $\Longrightarrow$ (ii): Suppose $S$ satisfies Property~\T1.  

\smallskip

\noindent \textsc{Claim:} \textit{If some point $w$ of $S$ is in the relative interior of any face $F$ (proper or not) of $\conv_\R(Z)$, then $Z \cap F \sbs S$ and the face $F$ is proper.}

\medskip

\noindent \textsl{Proof of Claim.} Toward a contradiction, assume there is $f_0 \in (Z \cap F) \setminus S$.  
The line passing through $f_0$ and $w$ must meet the relative boundary of $F$ on the side of $w$ opposite to $f_0$.  
Thus there exist vertices $f_1, f_2, \dots, f_k$ of the face $F$ such that $w$ belongs to the relative interior of the simplex with vertices $f_0,f_1, \dots,f_k$.  
Notice $\{f_0,f_1,f_2,\dots,f_k\} \sbs Z$ (because all vertices of $\conv_\R(Z)$ must be in $Z$).  
We split the analysis in the only two possible cases.
\vspace{-0,2cm}
\begin{description}
	\item[\rm Case 1:] $f_i$ is in $S$, for some $i \in \{1,2,\dots, k\}$.  
	Set $A=\{w,f_0,f_1,\dots, f_k\}$, and notice $f_i\in \ex(A) \cap S$ together with $w \in \ex(A \cap S) \setminus \ex(A)$.  
	This contradicts the assumption that $S$ satisfies \T1.
	\vspace{-0,2cm}
	\item[\rm Case 2:] $\{f_0,f_1,\dots,f_k\} \sbs Z \setminus S$.
	By our assumption $|S|\ge2$, there is $v \in S \setminus \{w\}$.  
	Consider two subcases for the possible position of the point $v$.  
	First, if $v \notin \conv_\R(\{f_0,f_1, f_2, \dots,f_k\})$, then we set $A=\{v,w,f_0, f_1, f_2, \dots,f_k\}$. 
	Notice $v \in \ex(A) \cap S$ and $w \in \ex(A \cap S) \setminus \ex(A)$, again a contradiction with $S$ satisfying \T1.  
	Second, if $v \in \conv_\R(\{f_0,f_1, f_2, \dots,f_k\})$, there is a point $x$ on the relative boundary of the simplex $T=\conv_\R(\{f_0,f_1,\dots,f_k\})$ such that $w \in \;]v,x[$.  
	Let now $A$ be formed by the points $v$, $w$ and the vertices of the minimal face of the simplex $T$ which contains $x$.  
	Again we get a contradiction because $v \in \ex(A) \cap S$ and $w \in \ex(A \cap S) \setminus \ex(A)$.  
\end{description}
\vspace{-0,2cm}
To complete the proof of the Claim, simply observe that $F$ must be proper, because otherwise we would have $Z=S$. 

\medskip

From the Claim, we derive that $S$ contains the intersection of $Z$ with any face of $\conv_\R(Z)$ containing in its relative interior at least one point of $S$.
Thus $S$ includes the intersection of $Z$ with the union of all such faces.  
The reverse inclusion also holds, because any point $w$ of $S$ belongs to both $Z$ and the relative interior of the smallest face of $\conv_\R(Z)$ containing $w$. 
\end{proof}

We are still missing a translation of Property~\T2 in Theorem~\ref{THM_characterizing_shrinkable_CGW}.
This translation appears to be of some interest in view of the fact that, along with the translation of \T1 obtained in Proposition~\ref{PROP_chrz_tracing_in_affine_geometry}, it would deliver a solution to Problem~\ref{PROB_1}.

The next example illustrates another type of obstruction to the affineness of a resolution.

\begin{example}\label{EX_affinity_is_unstable_under_affinity_is_unstable_under_resolutions_BIS}
Let   
$$
\begin{array}{lllll}
		X = \{1,2,3\}\,,& \GGG_X = 2^X \setminus \{\{1,3\}\}\,, \\
		Y_1 = \{a\} \,,& \GGG_1 = 2^{Y_1}\,, \\
		Y_2 = \{b\} \,,& \GGG_2 = 2^{Y_2}\,, \\
		Y_3 = \{c,d\} \,,& \GGG_1 = 2^{Y_3}\,. \\ 
\end{array}
$$
The resolution of $(X,\GGG_X)$ into $\{(Y_i,\GGG_i) \: \st \: i \in X\}$ is the convex geometry $(Z,\GGG_{Z}) = (X,\GGG_{X})\boxleft (Y_{i},\GGG_{i})_{i \in X}$, where $Z = \{a,b,c,d\}$, and 
$$
\GGG_Z = 2^Z \setminus \big\{ \{a,c\},\{a,d\} \big\}.
$$
Although all convex geometries $(X,\GGG_X)$ and $(Y_i,\GGG_i)$, $i = 1,2,3$, are affine, their resolution $(Z,\GGG_Z)$ is not. 
Indeed, we have $b\in\conv_{\GGG_Z}(\{a,c\}) \cap \conv_{\GGG_Z}(\{a,d\})$, along with $c\notin\conv_{\GGG_Z}(\{b,d\}$ and $d\notin\conv_{\GGG_Z}(\{b,c\}$, which is impossible in any affine geometry. 
(If $(Z,\GGG_Z)$ were affinely embedded, we would have in some real affine space $c$ and $d$ on the line through $a$ and $b$, on the side of $b$ opposite to $a$. However, this implies $c\in[b,d]$ or $d\in[b,c]$.)
\end{example}

The crucial assumptions in the last example are that $2$ lies between $1$ and $3$, and that the fiber $Y_3$ contains more than one element.  We generalize them in the next proposition (where $p$ plays the role of $2$). 

\begin{proposition}\label{PROP_pt_in_reint}
Suppose a resolution $(Z,\GGG_Z)=(X,\GGG_X) \boxleft (Y_x,\GGG_x)_{x \in X}$ of convex geometries is affine, with $\GGG_Z$ the geometry induced on the subset $Z$ of some real affine space $\R^d$.
Assume that the base contains elements $p$, $p_1$, \dots, $p_{n+1}$ such that $p \in \conv_{\GGG_X}(\{p_1$, $p_2$, \dots, $p_{n+1}\})$ and $p \notin \conv_{\GGG_X}(T)$ for any proper subset $T$ of $\{p_1$, $p_2$, \dots, $p_{n+1}\}$.  For $i=1,2, \ldots, n+1$, let $q_i$ be any point in the fiber $Y_{p_i}$.
Then all fibers $Y_{p_i}$ lie in the affine subspace of dimension $n$ generated by the points $q_1$, $q_2$, \dots, $q_{n+1}$, and so all fibers $(Y_{p_i},\GGG_{p_i})$ are isomorphic to convex geometries affinely embedded in a real affine space of dimension $n$.
\end{proposition}
   
\begin{proof}
As mentioned in Lemma~\ref{LEMMA_affine_are_hered}, the base $(X,\GGG_X)$ and all fibers $(Y_x,\GGG_{x})$ are also affine geometries. 
By Theorem~\ref{THM_resolution_of_convex_geometries}, the fiber $Y_p$ contains just one point, say $q$.  
Note that $q \in \conv_\R(\{q_1$, $q_2$, \dots, $q_{n+1}\})$ in $\R^d$ (the reason is that the projection on the base of the convex hull $\conv_{\GGG_Z}\{q_1$, $q_2$, \dots, $q_{n+1}\}$ in $Z$ must be convex in the base $X$ and at the same time contain $p_1$, $p_2$, \dots, $p_{n+1}$; thus by our assumptions the projection contains also $p$).  Moreover, $q$ cannot be in the convex hull of less than $n+1$ of the points $q_1$, $q_2$, \dots, $q_{n+1}$ (because the projection $p$ of $q$ does not lie in the convex hull in $X$ of less than $n+1$ of the $p_i$'s).  Thus in $\R^d$, the point $q$ is in the relative interior of the simplex with vertices $q_1$, $q_2$, \dots, $q_{n+1}$.
To derive the thesis for $i=1$ (the arguments are similar for the other values of $i$), note that $q_{1}$ lies in the affine hull of the points $q$, $q_2$, $q_3$, \dots, $q_{n}$.  As this result holds for any point in the fiber $Y_{p_1}$ in place of $q_{1}$, we deduce that the fiber $Y_{p_1}$ is included in the affine hull of $q$, $q_2$, $q_3$, \dots, $q_{n+1}$, which is the same as the affine hull of $q_1$, $q_2$, $q_3$, \dots, $q_{n+1}$.
\end{proof}

We know of several other necessary conditions for an affine convex geometry to be primitive, but none of them is both necessary and sufficient.  We leave Problems~\ref{PROB_1} and \ref{PROB_2}  unsolved.


\subsection{Ordinal Convex Geometries}
\label{SUBSECT_ordinal cg}

Recall from Example~\ref{EX_ordinal_Convex_Geometries} and Theorem~\ref{THM_chrz ordinal cg} that a convex geometry $(Z,\GGG_Z)$ is ordinal if and only if $\GGG_Z$ is closed under union, or, equivalently, $\GGG_Z$ consists of all ideals of some unique partial order $\le$ on $Z$ (the partial order \textsl{associated} to $\GGG_Z$).
Remark~\ref{REM_extreme_operator_for_ordinal_CG_is_maximization} readily yields that the equivalence
\begin{equation}\label{EQ_order_ext_bis}
\lnot (z < z') \quad  \iff \quad z \in \ex_{\GGG_Z}(\{z,z'\})
\end{equation}
holds for all $z,z' \in Z$. 
Here we show that (1)~a resolution of ordinal convex geometries is always ordinal, and (2)~its associated partial order is the `resolution' (as in the next definition) of the partial orders associated to the base and the fibers. 

\begin{definition}\label{DEF_resolution_of_relations}
Let $X$ be a finite \textsl{base} set, and $\{Y_x \st x \in X\}$ a family of finite, pairwise disjoint \textsl{fiber} sets disjoint from the base set.   
Furthermore, let $R_X$ be a binary relation on $X$, and $R_x$ a binary relation on $Y_x$ for each $x \in X$. 
Set $Z := \bigcup_{x \in X} Y_x$, and call \textsl{projection} the mapping $\pi \colon Z \to X$, with $\pi(z)=x$ when $z \in Y_x$.
The \textsl{resolution of $(X,R_X)$ into $\{(Y_x,R_x) \:\st\: x \in X\}$} is the pair $(Z,R_Z)$, where $R_Z$  is the binary relation on $Z$ defined by  
\begin{equation} \label{EQ_resolution_relations}   
z R_Z z' \; \iff  \;
\begin{cases}
  \text{either} & \!\!\!(\exists x \in X) \; (z,z' \in Y_x  \wedge z R_x z') \,,\\
  \text{or}  & \!\!\!(\exists x,x' \in X) \; (x \neq x' \wedge z \in Y_x \wedge z' \in Y_{x'} \wedge x R_X x')
\end{cases}
\end{equation}
for all $z,z' \in Z$. 
With a slight abuse of terminology, we shall also say that $R_Z$ is the resolution of $R_X$ into the family $\{R_x \st x \in X\}$. 
We use a notation similar to the one employed for convex geometries, namely
\begin{equation*} 
	(Z,R_Z) = (X,R_X) \boxslash (Y_x,R_x)_{x \in X}.	
\end{equation*}
A binary relation on a finite set is \textsl{primitive} when it cannot be obtained as a nontrivial resolution of relations, and is \textsl{resolvable} otherwise.\footnote{As usual, \textsl{nontrivial} means that the base has more than one element, and there is at least a fiber that has more than one element.} 
\end{definition}

\begin{remark} \rm \label{REM_terminology_for_resolutions_of_binary_relations}Resolutions of binary relations are well-known, often under a different name: see, for instance, \cite{Dorfler1971}.
For the special case of partial orders, they are called \textsl{sums} by \citet{Hiraguchi1951}, \textsl{lexicographic sums} by \citet[page~24]{Trotter1992}, and \textsl{ordered sums} by \citet[page~85]{Harzheim2005}. 
Observe that \citet[page~8]{Bang-Jensen_Gutin2001} use the term `composition' in place of `resolution'.  In this paper, we employ the term `resolution' for binary relations not only to avoid confusion, but also in an attempt to use a common name for the codification of the same concept in different mathematical settings.
\end{remark}

It is well-known that the resolution of binary relations is a partial order exactly when the base relations and the fiber relations are all partial orders: see \cite{Hiraguchi1951}, \cite{Trotter1992}, or \cite{Harzheim2005}. 
The main result of this section (Theorem~\ref{THM_composition_of_partial_orders}) proves two things: (1)~a resolution of convex geometries is ordinal if and only if so are its base and its fibers; (2)~there is a tight connection between resolutions of ordinal convex geometries and resolutions of partial orders.

\begin{theorem} \label{THM_composition_of_partial_orders}
A resolution of convex geometries is an ordinal convex geometry if and only if its base and all its fibers are ordinal convex geometries.
Furthermore, the partial order associated to the resolved convex geometry is equal to the resolution of the partial order associated to the base into the family of partial orders associated to the fibers. 
\end{theorem}

\begin{proof}   
Let $(Z,\GGG_Z) = (X,\GGG_X) \boxleft (Y_x,\GGG_x)_{x \in X}$ be a resolution of convex geometries.

If the convex geometry $(Z,\GGG_Z)$ is ordinal, equivalently $\GGG_Z$ is closed under union (Theorem~\ref{THM_chrz ordinal cg}),  Remark~\ref{REM_trivia_on_resolutions_cg} implies that all geometries $(X,\GGG_X)$ and $(Y_x,\GGG_x)$, for $x \in X$, are also ordinal (because a subgeometry of an ordinal geometry is itself ordinal).

\smallskip 

To prove the converse, suppose now that $(X,\GGG_X)$ and $(Y_x,\GGG_x)$, for $x \in X$, are all ordinal convex geometries.
By Theorem~\ref{THM_chrz ordinal cg}, it suffices to show that $\GGG_Z$ is closed under union.
Let $B,C \in \GGG_Z$.
We shall prove that $D = B\cup C$ satisfies Requirements~\textsf{(R1)--(R3)} in Definition~\ref{DEF_resolution_of_cg}.  
\begin{itemize}
  \item[\textsf{(R1)}] Requirement~\textsf{(R1)} applied to $B$ and $C$ yields $\pi(B),\pi(C) \in \GGG_X$, hence $\pi(D) = \pi(B) \cup \pi(C) \in \GGG_X$, because $\GGG_X$ is closed under union. 
  \item[\textsf{(R2)}] Let $x \in \pi(D)$. Without loss of generality, assume that $x \in \pi(B) \cap \pi(D)$ (indeed, if $x$ belongs to exactly one between $\pi(B)$ and $\pi(C)$, the result is trivial). Now Requirement~\textsf{(R2)} applied to $B$ and $C$ yields $B \cap Y_x, C \cap Y_x \in \GGG_x$.  Since $D \cap Y_x = (B \cap Y_x) \cup (C \cap Y_x)$ and $\GGG_x$ is closed under union by assumption, we derive $D \cap Y_x \in \GGG_x$, as claimed.   
  \item[\textsf{(R3)}] Let $x \in \pi(D) \setminus \ex_{\GGG_X}(\pi(D)) = (\pi(B) \cup \pi(C)) \setminus \ex_{\GGG_X} (\pi(B) \cup \pi(C))$.
It follows that 
$x \in \pi(B) \setminus \ex_{\GGG_X}(\pi(B))$ or $x \in \pi(C) \setminus \ex_{\GGG_X}(\pi(C))$ holds. 
By Requirement~\textsf{(R3)} applied to $B$ or $C$, we get  $Y_x \sbs B$ or $Y_x \sbs C$, hence $Y_x \sbs B \cup C = D$. 
Thus, $D$ satisfies~\textsf{(R3)}, too. 
\end{itemize}

Next, we prove the second assertion. 
Let $\le_X$, $\le_x$, and $\le_Z$ be the partial orders associated to the ordinal convex geometries $(X,\GGG_X)$, $(Y_x,\GGG_x)$, and $(Z,\GGG_Z)$, respectively.  
We show that $(Z,\le_Z)$ is the resolution (in terms of relations) of $(X,\le_X)$ into the family $(Y_x,\leq_x)_{x \in X}$, that is, 
$$
(Z,\leq_Z) = (X,\leq_X) \boxslash (Y_x,\leq_x)_{x \in X}\,.
$$
Indeed, in view of the equation \eqref{EQ_ex_Z} in Lemma~\ref{LEMMA_conv_ex}, the equivalence~\eqref{EQ_order_ext_bis}, and the equivalence~\eqref{EQ_resolution_relations} in Definition~\ref{DEF_resolution_of_relations}, we have, for all $z,z' \in Z$, 
\begin{alignat*}{3}
\lnot(z <_Z \!z') \!&\iff &&\;  z \in \ex_{\GGG_Z}(\{z,z'\})\\
&\iff&& z \in \bigcup_{x \, \in \,\ex_{\GGG_X}(\pi(\{z,z'\}))} \ex_{\GGG_x}(\{z,z'\} \cap Y_x)\\
&\iff &&\; 
\!\!\!\begin{array}[t]{l}
(\exists x\in X) (z,z'\in Y_x \land z \in \ex_{\GGG_x}(\{z,z'\})) \quad \text{or}\\
(\exists x,x'\in X) (x\neq x' \land z\in Y_x \land z'\in Y_{x'} \land \pi(z) \in \ex_{\GGG_X}(\pi(\{z,z'\})))
\end{array}
\\
&\iff && \;
\!\!\!\begin{array}[t]{l}
(\exists x\in X) (z,z'\in Y_x \land \lnot (z <_x z')) \quad \text{or}\\
(\exists x,x'\in X) (x\neq x' \land z\in Y_x \land z'\in Y_{x'} \land \lnot (\pi(z) <_X \pi(z'))).
\end{array}
\end{alignat*} 
We conclude that $z <_Z z'$ does not hold if and only if the pair $(z,z')$ does not belong to the resolution of the partial order $\le_X$ into the family $\{\le_x \st x \in X\}$.
This completes the proof. 
\end{proof}

The following consequence of Theorem~\ref{THM_composition_of_partial_orders} is immediate:  

\begin{corollary}\label{COR:primitive ordinal}
An ordinal convex geometry is primitive if and only its associated partial order is primitive.
\end{corollary}

For information on primitive posets, we refer the reader to~\cite{Schmerl_Trotter1993} or \cite{Boudabbous_Zaguia_Zaguia2010}.  
The concept of primitivity applies to more general relational structures: see~\cite{Ille2005a} for a survey.

The shrinkable sets of an ordinal convex geometry $(Z,\GGG)$ are exactly the autonomous sets of the associated partial order $\le$, where (see \citealp{Schroder2016}) $S \subseteq Z$ is \textsl{autonomous} if for all $s, s'\in S$ and $z \in Z \setminus S$,   
$$
s \le z \;\implies \; s' \le z \qquad\text{and}\qquad z \le s \;\implies \;z \le s'. 
$$

Finally, observe that Theorem~\ref{THM_composition_of_partial_orders} does not hold for compositions:  indeed, Example~\ref{ex_compos_not_convex_geometry} exhibits a composition of ordinal convex geometries that fails to be a convex geometry.\footnote{A close link between compositions of set systems and compositions of posets results from attaching to a poset its set of chains, as explained in \cite{Mohring_Radermacher1984}.}


\section{Primitivity of Small Convex Geometries}
\label{SECT:cg on at most 4 pts}

Here we determine all primitive convex geometries on at most four elements. 
Observe preliminarily that our classification task is simple for the special case of ordinal convex geometries. 
In fact, by Corollary~\ref{COR:primitive ordinal}, to test whether an ordinal convex geometry is primitive, it suffices to check whether its associated poset is primitive (as a poset), which in turn amounts to investigate whether the poset has an autonomous subset.

To start, note that all convex geometries on one or two elements are primitive.  The next proposition inspects which convex geometries on three and four elements are primitive, also determining whether they are ordinal or affine.
A list of all $34$~convex geometries on four elements appear in \citet{Merckx2013}, and their number is confirmed in \citet{Uznanski2013}.

\begin{proposition} \label{prop_geometries_on_4}
Up to isomorphisms, there are: 
  \begin{enumerate}[\rm (i)]
    \item $6$ convex geometries on three elements, of which $1$ is primitive and non-ordinal, and $5$ are resolvable and ordinal;
    \item $34$ convex geometries on four elements, $12$ of which are primitive; among the primitive ones, $1$ is ordinal and $2$ are affine. 
    \end{enumerate}
\end{proposition}

\begin{proof}    
(i) On three elements there are, up to isomorphisms, $6$ geometries, which are listed in Example~\ref{EX_all_CGs_on_3_elements_up_to_iso}.  
By using Theorem~\ref{THM_chrz ordinal cg},  one can readily check that exactly $5$ of them are ordinal: in fact,  there is only one convex geometry that is not closed under union, namely $\GGG_5$. 
All $5$ corresponding posets on three elements are resolvable, hence also the associated ordinal convex geometries are resolvable. 
Furthermore, the unique non-ordinal geometry $\GGG_5$ is primitive, since otherwise its base and fibers would be ordinal, and so $\GGG_5$ itself would also be ordinal by Theorem~\ref{THM_composition_of_partial_orders}.  

\medskip

(ii) On four elements, there are $34$~convex geometries, $16$ of which are ordinal.  
The $16$~posets on four elements are listed, for instance, in \cite{Monteiro_Savini_Viglizzo2017} and \cite{Steinbach1990}; only one of these posets is primitive (it is the `N-poset'). 
To find out how many of the $18$~non-ordinal convex geometries on four elements are primitive, we rather look for the number of resolvable ones.  

By Theorem~\ref{THM_composition_of_partial_orders}, non-ordinal resolutions have either a non-ordinal base or a non-ordinal fiber (or both).  
Furthermore, all fibers of a nontrivial resolution on four elements have size at most three.   
Thus there are only two cases: 
\begin{itemize}
	\item[(a)] the base is the unique non-ordinal geometry $\GGG_5$ on three elements, and the three fibers have one, one, and two elements, respectively;
	\item[(b)] the base has two elements and one fiber is $\GGG_5$. 
\end{itemize}

Taking into account the automorphisms of small convex geometries, we are left with $7$~possible resolutions, of which $4$ are of type (a), and $3$ of type (b).  
It is simple to construct these $7$ resolvable geometries, and check that they are pairwise non-isomorphic.  
We conclude that among the $18$ non-ordinal convex geometries on four elements, $11$ are primitive. 
Moreover, exactly $2$ of these $11$ primitive geometries are affine.
In fact, there are exactly $4$~affine convex geometries on four points, which are those induced on the subsets of the real affine plane shown in Figure~\ref{fig_four_subsets}; only the first and the fourth produce a primitive convex geometry. 
\end{proof}

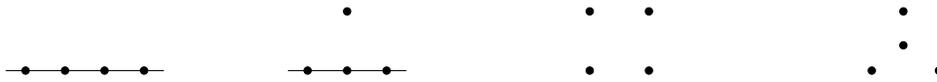
\begin{figure}[ht]
\begin{center}
\begin{tikzpicture}[scale=0.75]
\tikzstyle{point}=[circle,draw,fill=black, scale=0.25]

\begin{scope}[scale=0.7]
\node[point] at (0,0){};\node[point] at (1,0){};
\node[point] at (2,0){};\node[point] at (3,0){};
\draw (-0.5,0)--(3.5,0);
\end{scope}

\begin{scope}[xshift=5cm,scale=0.7]
\node[point] at (0,0){};\node[point] at (1,0){};
\node[point] at (2,0){};\node[point] at (1,1.5){};
\draw (-0.5,0)--(2.5,0);
\end{scope}

\begin{scope}[xshift=10cm,scale=0.7]
\node[point] at (0,0){};\node[point] at (1.5,0){};
\node[point] at (0,1.5){};\node[point] at (1.5,1.5){};
\end{scope}

\begin{scope}[xshift=15cm,scale=0.7]
\node[point] at (0,0){};\node[point] at (1.7,0){};
\node[point] at (0.8,0.64){};\node[point] at (0.8,1.5){};
\end{scope}

\end{tikzpicture}
\end{center}
\caption{Four subsets of the real affine plane (see the proof of Proposition~\ref{prop_geometries_on_4}).\label{fig_four_subsets}}
\end{figure}


\section{Future Work and Open Problems}
\label{SECT:open problems}

Here we list a few natural problems (they might be easy or difficult). 
By a class of convex geometries, we mean a class closed under taking isomorphic images.
\begin{enumerate}[1.]
	\item For a property (P) shared by some convex geometries, consider the following two assertions:
	\begin{enumerate}[(i)]
		\item if the base and the fibers of a resolution all satisfy (P), then the resolution also satisfies (P);
		\item if a resolution satisfies (P), then its base and its fibers  satisfy (P).
	\end{enumerate}
We say that the property~(P) is \textsl{forward stable} when (i) is true, \textsl{backward stable} when (ii) is true, and  \textsl{stable} when both (i) and (ii) are true.
For instance, ordinality of a convex geometry is a stable property (Theorem~\ref{THM_composition_of_partial_orders}), whereas affineness is a backward stable property (Lemma~\ref{LEMMA_affine_are_hered}) that fails to be forward stable   (Example~\ref{EX_affinity_is_unstable_under_resolutions}).
An interesting problem consists of determining which (additional) properties of convex geometries considered in the literature are preserved by resolutions, in particular which of the known families of convex geometries are stable under resolutions (see \citealp{Goecke_Korte_Lovasz1989} for several types of such families).
\cite{Carpentiere2019} shows that neither monophonic convex geometries nor geodetic convexity geometries form a stable family (for monophonic vs geodetic convex sets in graphs, see \citealp{Farber_Jamison1986}).

\item Any class $\CCC$ of convex geometries is included in a smallest class $\SSS$ of convex geometries forward stable under resolutions.  When $\CCC$ itself is not forward stable, $\SSS$ differs from $\CCC$.  Characterize $\SSS$ when $\CCC$ is the class of affine convex geometries, and also for other nonstable classes $\CCC$. 
\item Design a non-na\"\i ve algorithm to test whether a given convex geometry is primitive, and another one to generate the primitive convex geometries on small numbers of elements. 
\cite{Enright2001} investigates various encodings of convex geometries.  
\cite{Uznanski2013} discusses a code generating all convex geometries up to seven elements (reporting ingegneous programming efforts).
\item As it is the case for many classes of structures with respect to compositions (see \citealp{Mohring_Radermacher1984}), does the fraction of primitive convex geometries on $n$ elements tend to one?
(Notice that there are two questions here,  one for convex geometries on labeled sets and one for convex geometries up to isomorphisms.)  
About the asymptotic number of labelled convex geometries, see \cite{Echenique2007} and \cite{Monjardet2008}.
For examples of recent results concerning prime structures, see \cite{Boudabbous_Ille2011} for binary relations, \cite{Guillet_Leblet_Rampon2017} for posets,
\cite{Ille_Villemaire2014} as well as  \cite{Chudnovsky_Kim_Oum_Seymour2016} for graphs.
\item As explained in Section~\ref{SECT:resolutions cg}, our definition of resolutions of convex geometries was inspired from the one of resolutions of (path-independent) choice spaces in \cite{Cantone_Giarlotta_Watson2020+}. On the other hand, \cite{Ando2006} (see also \citealp{Danilov_Koshevoy2009}) relate so-called quasi-choice spaces to closure spaces.  Investigate (a variant of) resolutions for those structures. 
\item When a convex geometry is expressed as a resolution, the fibers could again be resolvable, and again their fibers, etc.  This leads to a notion of `deresolution tree',  which is similar to that of  a decomposition tree based on  hypergraph compounds \citep{Chein_Habib_Maurer1981} or set-system compositions \citep{Mohring_Radermacher1984,Mohring1985a}.
A manuscript under preparation contains a uniqueness result for a well-defined deresolution tree of a convex geometry, and even of a choice space.  
`Strong shrinkable sets' are the main tools, similar to strong modules for decomposition trees (for recent work on the latter, see, for instance,  \citealp{Bonizzoni_Della-Vedova1999,Foldes_Radeleczki2016,Habib_Montgolfier_Mouatadid_Zou2019}).  

\item \cite{Kashiwabara_Nakamura_Okamoto2005} show that any convex geometry can be obtained by a construction that generalizes the one for affine convex geometries (which is obtained by taking $Q=\es$ below).  Specifically, given two finite subsets $P$ and $Q$ in some real affine space $\R^d$ with the property that $P \neq \es$ and $P \cap \conv_\R(Q) = \es$, let 
$$
\LLL \;:=\; \big\{G \in 2^P \st \conv_\R(G \cup Q) \cap P = G \big\}. 
$$
Then $(P,\LLL)$ is a convex geometry, and moreover any convex geometry is isomorphic to such a geometry.  The following are extensions of Problems~\ref{PROB_1} and \ref{PROB_2}: (i)~characterize the shrinkable subsets of the convex geometry $(P,\LLL)$ in terms of $P$ and $Q$; (ii)~characterize the pairs $(P,Q)$ for which the convex geometry $(P,\LLL)$ is primitive. 
\item There are some definitions of infinite convex geometries in the literature: see, for instance,  \cite{Adaricheva2014b},
\cite{Adaricheva_Nation2016d}, \cite{Jamison-Waldner1982}, \cite{Mao2017}, \cite{Mao_Liu2012}, \cite{Marti_Pinosio2020}, \cite{Wahl2001}.  In view of these notions, extending the investigation of resolutions to infinite convex geometries appears to be of some interest. 
\end{enumerate} 


\subsection*{Acknowledgments}
The authors wish to thank Davide Carpentiere for several useful discussions on the topic, and Pierre Ille for information on primitive posets. 


\end{document}